\documentclass[twoside, 11pt]{article}
\usepackage{amssymb, amsmath, mathrsfs, amsthm}%, wasysym}
\usepackage{graphicx}
\usepackage{color,pict2e}
\usepackage[top=2cm, bottom=2cm, left=2cm, right=2cm]{geometry}
\usepackage{float, caption, subcaption}
\usepackage{amsmath,amsthm,amsfonts,amssymb,amscd,mathrsfs}

\input{epsf}

\newcommand{\bd}{\begin{description}}
\newcommand{\ed}{\end{description}}
\newcommand{\bi}{\begin{itemize}}
\newcommand{\ei}{\end{itemize}}
\newcommand{\be}{\begin{enumerate}}
\newcommand{\ee}{\end{enumerate}}
\newcommand{\beq}{\begin{equation}}
\newcommand{\eeq}{\end{equation}}
\newcommand{\beqs}{\begin{eqnarray*}}
\newcommand{\eeqs}{\end{eqnarray*}}

\definecolor{DarkGreen}{rgb}{0.2, 0.6, 0.3}

\newcommand{\sat}{{\rm sat}}
\newcommand{\Sat}{{\rm Sat}}

\newcommand{\diam}{{\rm diam}}

\newtheorem{theorem}{Theorem}[section]

\newtheorem{lemma}[theorem]{Lemma}

\newtheorem{claim}{Claim}

\newtheorem{proposition}[theorem]{Proposition}

\newtheorem{problem}{Problem}

\begin{document}
\title{\textbf{Some results on minimum saturated graphs} %\footnote{Supported by the NationalScience Foundation of China (Nos. 11601254, and 11551001).}
 }

\author{ Chenke Zhang\footnote{School of Mathematics and Statistics, Shandong University of Technology,
			Zibo 255000, China. {\tt
		zhangchenke2024@163.com}}, \ \
        Qing Cui\footnote{School of Mathematics
Nanjing University of Aeronautics and Astronautics,
Nanjing 210016,  China. {\tt cui@nuaa.edu.cn}},\ \  
	Jinze Hu\footnote{Center for Discrete Mathematics, Fuzhou University, Fuzhou, Fujian 350116, China.{\tt hujinzeh@163.com}},\ \  Erfei Yue\footnote{HUN-REN Alfr\'{e}d R\'{e}nyi Institute of Mathematics \& E\"{o}tv\"{o}s Lor\'{a}nd University, Budapest, Hungary. {\tt yef9262@mail.bnu.edu.cn}},\ \  Shengjin Ji \  \footnote{School of Mathematics and Statistics, Shandong University of Technology,
			Zibo 255000, China. {\tt jishengjin@sdut.edu.cn}}
    \footnote{Corresponding author: Shengjin Ji}
			}

\date{}
\maketitle
%\vspace{-5mm}
%\begin{center}
%Dedicated to Professor Eddie Cheng on the occasion of his 60th birthday.
%\end{center}
%\vspace{-3pt}
%\begin{center}
%	\small  School of Mathematics and statistics, Shandong University of Technology,
%	\\ Zibo, Shandong 255049, China.
%\end{center}
%\vspace{1mm}

\begin{abstract}
Let $G$ be a graph and $\mathcal{F}$ be a family of graphs. We say a graph $G$ is $\mathcal{F}$-saturated if $G$ does not contain any member in $\mathcal{F}$  and for any $e\in E(\overline{G})$, $G+e$ creates a copy of some member in $ \mathcal{F}$. The saturation number of $\mathcal{F}$  is the minimum number of edges of an $\mathcal{F}$-saturated  graphs with $n$ vertices, denoted by $\sat(n,\mathcal{F})$. If $\mathcal{F}=\{F\}$, then we write it as $\sat(n,F)$ for short.  In this paper, we determine the exact value of $\sat(n,\{K_3,P_k\})$, and as its application, we obtain two bounds of $\sat(n,K_3\cup P_k)$ for $k\ge 10$ and  sufficiently large $n$. Furthermore, $\sat(n,K_1\lor F)$ is determined, where $F$ is a linear forest without isolated vertices.
\\[2mm]
\end{abstract}

\section{Introduction}

For a given family $\mathcal{F}$ of graphs, we say a graph $G$ is $\mathcal{F}$-free, if $G$ does not contain any member in  $\mathcal{F}$. We say a graph $G$ is $\mathcal{F}$-saturated if $G$ is $\mathcal{F}$-free, and for any $e\in E(\overline{G})$, $G+e$ creates a copy of some member in $\mathcal{F}$. The saturation number of $\mathcal{F}$ is defined as $\sat(n,\mathcal{F})=\min\{e(G): G$ is  $\mathcal{F}$-saturated  and  $|G|=n\}$ and the extremal graphs of $\mathcal{F}$ are belonging to  $\Sat(n,\mathcal{F})=\{G: G$ is $\mathcal{F}$-saturated with $e(G)=\sat(n,\mathcal{F})\}$.  We substitute  $\mathcal{F}$-free,  $\mathcal{F}$-saturated,  $\sat(n,\mathcal{F})$ and $\Sat(n,\mathcal{F})$ 
 with $F$-free, $F$-saturated, $\sat(n,F)$ and $\Sat(n,F)$, respectively if $\mathcal{F}=\{F\}$.

Saturation number was first introduced by Erd\H{o}s et al. \cite{EHM64} who showed that $\sat(n,K_p )=(p-2)(n-p+2)+\binom{p-2}{2}$ and $K_{p-2}\vee \overline{K}_{n-p+2}$ is the unique minimum $K_p$-saturated graphs with $n$ vertices.  K\'{a}szonyi and Tuza \cite{KT86} determined the saturation numbers of a star, a path and an $m$-matching. Furthermore, they proved that saturation number is bounded by a linear function of $n$. For cycles, we refer  to \cite{ C11,LSWZ21,Y25,MHHG21,MPT25,T89}. For a disjoint union of cliques, Faudree et al. \cite{FFGJ09} determined $\sat(n,tK_p)$, $\sat(n,K_p\cup K_q)$ and $\sat(n,F_{t,p,l})$. Chen and Yuan \cite{CY23} determined the saturation number for $K_p\cup(t-1)K_q$, and the extremal graph for $K_p\cup2K_q$($2 \le p<q$). Moreover, the saturation number and extremal graph for $K_p\cup K_q\cup K_r$ ($r \ge p+q$) are completely determined. Later, Zhu et al. \cite{ZHH24} resolved a conjecture in \cite{CY23} by determining $
\Sat(n,K_p\cup(t-1)K_q)$ for every $2\le p \le q$ and $t\ge 2$. %Li et al. \cite{LHHZ25} showed the saturation number for unions of four cliques.
For a linear forest $F$, Chen et al. \cite{CFFGJM17} investigated the saturation numbers for forests and provided the upper and lower bounds on $\sat(n,H)$ with $H\in \{F,tP_k,P_k\cup P_l\}$. Furthermore, they obtained the exact values of $\sat(n,P_m\cup tP_2)$ with $m\in \{3,4\}$. 
So far, for $m\in \{5,6,7\}$,
$\sat(n,P_m\cup tP_2)$ are also determined, see \cite{FW15,Y23,ZHHZ24}. In addition, two results on $\sat(n, tP_3)$ are presented in \cite{CLLYZ23,HLL23}. For more other saturated results, we refer to a survey  \cite{CFF21}.

%Fan and Wang \cite{FW15} proved that $\sat(n,P_5\cup tP_2)=\smin \left\{\lceil \frac{5n-4}{6} \rceil ,3t+12\right\}$ for $n\ge 3t+8$ with the extremal graphs $K_6\cup (t-1) K_3\cup \overline{K_{n-3t-3}}$ for $n>(8t+76)/5$. Moreover, Yan \cite{Y23} showed that $\sat(n,P_6\cup tP_2)=\smin\left\{n-\lfloor \frac{n}{10} \rfloor, 3t+18\right\}$ with the extremal graphs $K_7\cup \left( t-1 \right) K_3\cup \overline{K_{n-3t-4}}$ for $n>10t/3+20$.  K\'{a}szonyi and Tuza \cite{KT86} gave the upper bound of $sat(n,F)$ for an arbitrary $F$. Cameron and Puleo \cite{CP22} gave a lower bound of the saturation number for an arbitrary graph. For more results on the saturation number, we can follow a survey \cite{CFF21}.

  Recently, the saturation number of the disjoint union of a clique and a path has been studied. %Hu et al. \cite{HJZ24} showed the exact value of $\sat(n,K_p\cup P_3)$.
  Li and Xu \cite{LX23} studied connected $K_3\cup P_k$-saturated graphs for $k\ge 4$ and posed a problem whether the size of the minimum  connected $K_3\cup P_k$-saturated graphs equals to $n+2$. Hu et al. \cite{HJ25} gave a positive answer of the problem under the condition sufficiently large $n$ and $k\ge 4$,  furthermore, they gave an upper bound of $\sat(n,K_3\cup P_k)$ for integer $k\ge 6$.
 % $\sat(n,P_k\cup K_3)\le n-\lfloor \frac{n-n_{0}^{'}}{a_{k}^{'}} \rfloor +3( \frac{k}{2}-1 ) $, if $k$ is even, and $\sat(n,P_k\cup K_3)\le n-\lfloor \frac{n-n_{0}^{'}}{a_{k}^{'}} \rfloor +3( \frac{k}{2}-1) +1$, if $k$ is odd.} 
  
  In this paper, we are interested in the saturation number of $K_3\cup P_k$. In fact, We will research $\sat(n,K_3\cup P_k)$  through establishing its relationship with $\sat( n,\{K_3,P_k\})$.  We first obtain the following result, where $a_k^1$ is defined in next section.
\begin{theorem}\label{saturationnumber2}
If $n\ge a^1_k$ and $k\ge 10$, then $\sat( n,\{K_3,P_k\}) =n-\lfloor n/a^1_k \rfloor$.
\end{theorem} Based on the result above, we can deduce two bounds on $\sat( n,K_3\cup P_k)$.
\begin{theorem}\label{k3vpk}
	For $k\ge 10$ and $n$ sufficiently large, we have that 
    \begin{equation}\label{twobounds}
    2+\sat( n,\{ K_3,P_k \} )\le \sat( n,K_3\cup P_k ) \le 6+\sat( n,\{ K_3,P_k \} ).
    \end{equation}
\end{theorem}
\noindent In fact,  the upper bound of $\sat( n,K_3\cup P_k )$ in Relation (\ref{twobounds}) is better than the upper bound in [\cite{HJ25}, Theorem 2.10].
   We are also interested in saturation number of the join of graphs.  K\'{a}szonyi and Tuza \cite{KT86} showed that $G$ is $F'$-saturated if and only if $G\setminus\{v^*\}$ is $F$-saturated, where $G$ has some center vertex $v^*$ and $F'$ has a  center vertex $v^{*}_1$ such that $F=F'\setminus \{v^*_1\}$. Cameron and Puleo \cite{CP22} showed that $\sat(n,F')\le(n-1)+\sat(n-1,F)$ for all $n>|V(F)|$. A natural problem is to find all graphs such that the equality  holds.
\begin{problem}\label{p1}
	For $n$ sufficiently large, determine the graph family $\mathcal{F}$ such that for each $F\in \mathcal{F}$ we have
    \begin{equation}\label{p11}
       \sat(n,K_1\lor F)=n-1+\sat(n-1,F).   
    \end{equation}  
\end{problem}

\noindent Recently, Hu et al. \cite{HSQ25} studied Problem \ref{p1}
 and  confirmed it for $F\cong P_t$ with $t\ge 5$ and  sufficiently large $n$. Song et al. \cite{SHJC25}  confirmed Problem \ref{p1} for $F\cong C_4$ and determined all minimum saturated graphs.  Hu et al. \cite{SZY25} showed that $\sat(n,K_s\lor F) =\tbinom{s}{2} +s(n-s) +\sat( n-s,F) $ for $n\ge 3s^2-s+2\sat(n-s,F)+1$, where $F$ is a graph without isolated vertex. Qiu et al. \cite{YZMY25} got an observation that in the above result, the restriction condition on $n$ implies that $F$ contains isolated edges. Moreover, they solved Problem \ref{p1} for the case $F\cong C_l$ with  $l\ge8$.  Note that we usually call $K_1\lor C_l$ a \emph{wheel} graph for $l\ge 3$.
We will research the problem for the case that $F$ is a linear forest with isolated vertices, and obtain the following result. 
\begin{theorem}\label{saturationjoin}
	Let $G$ be a minimum $K_1\lor F$-saturated graph, then $e(G)=(n-1)+\sat(n-1,F)$ for sufficiently large $n$ and $\Sat(n, K_1\lor F)=\{K_1\lor H:  H\, \emph{is a minimum}\, $F$\emph{-saturated graph}\}$.
\end{theorem}

For convenience, we now define some terminology and notation. All graphs considered in the paper are finite and simple.  For a given graph $G$, let $V(G)$ and $E(G)$ be the \textit{vertex-set} and \textit{edge-set} of $G$, respectively. %Denote $K_k$, $C_k$, $P_k$ and $S_k$ a complete graph, a cycle, a path and a star with $k$, respectively. 
Let  $G[S]$ be the  subgraph of $G$ induced by $S$ if $S\subseteq V(G)$.
 For any $v\in V(G)$, let $N_G(v)$ denote the set of vertices adjacent to $v$ and $N_G[v]=N_G(v)\cup \{v\}$. The \emph{degree} of a vertex $v$ is $|N_G(v)|$ and let $\delta( G)$ and $\varDelta ( G )$ denote the minimum and maximum degree of a vertex in $G$, respectively. A vertex $v^*$ of $G$ on order $n$ is called a \textit{center vertex} if $d(v^*)=n-1$. A graph is said to be \emph{connected}, if for every pair of vertices there is a path joining them, disconnected otherwise. A maximal connected subgraph of $G$ is called a component of $G$. We call vertex $v\in V(G)$ a cut vertex if removing $v$ from $G$ increases components.  The \emph{connectivity} $\kappa ( G)$ of $G$ is the minimum size of vertex subset $S$ such that $G-S$ is disconnected or has only one vertex. 
 For vertices $u,v\in V(G)$, the \emph{distance} $d_G( u,v)$ of $u$ and $v$ is the length of a shortest path from $u$ to $v$. The \emph{diameter} $\diam(G)$ of $G$ is the largest distance over all pairs of vertices of $G$. 
 Given any two vertex-disjoint graphs $G$ and $H$, let $G\cup H$ be the \emph{union} of $G$ and $H$ with vertex set $V(G)\cup V(H)$ and edge set $E(G)\cup E(H)$, and let $G\lor H$ be the \emph{join} of  $G$ and $H$ obtained by adding all edges between  $G$ and $H$ in $G\cup H$.   For other notions not defined here, refer to  \cite{B78,BM08}.

The remainder of this paper is organized as follows. In Section $2$, we introduce  some basic results which will be used in the sequel. In Section $3$, we  show the exact value of $\sat(n,\{K_3,P_k\})$ with $k\ge 10$, and give an upper bound and a lower bound of $\sat(n,K_3\cup P_k)$. In Section $4$, we determine the saturation number of $K_1\lor F$ and characterize all extremal graphs. In Section $5$,  we begin with a brief summary and then pose an unsolved problem. Furthermore, the minimum $\{K_3,P_k\}$-saturated trees are also presented with $k\le 9$ for the sake of completeness.

\section{Preliminary}
 We begin this section by introducing three types of trees, and then present some basic conclusions on saturation numbers of $\{K_3,P_k\}$. 

 {\bf Layer:} In order to describe clearly  the  structure of a tree, we introduce the notation of ``layer" of a tree. Let $T$ be tree with $diam(T)=s\ge 2$. Hence, $T$ has a longest path of order $s+1$, say $P_{s+1}=v_1v_2\cdots v_{s+1}$. We call the middle two vertices (or one vertex) belonging to the $1$-layer of $T$,
and all other vertices belonging to the $i$-layer if their distance to the $1$-layer is $i-1$ for $2\le i\le \lceil\tfrac{s+1}{2}\rceil$.
More formally, we use  $l(v)$ denote the the layer number of every vertex $v\in V(T)$,  in other words, $l(v)=i$ if and only if  $v$ is lying on the $i$-layer of  $T$. We observe that all vertices of a tree with diameter $s$ can be partitioned  into the $ \lceil\tfrac{s+1}{2}\rceil$ layers.

 \vspace{3mm}
   \begin{figure}[h]\label{1}
  \begin{center}
\begin{picture}(292.7,62.3)\linethickness{0.8pt}
\Line(45,47.6)(29.7,35.6)
\Line(29.7,35.6)(17,23.6)
\Line(29.7,35.6)(28.3,23.6)
\Line(48.3,24)(56,10)
\Line(48.3,24)(48.3,10)
\Line(38.3,24)(40,10)
\Line(28.3,23.6)(24,10)
\Line(28.3,23.6)(16,10)
\Line(17,23.6)(0,10)
\Line(44.3,35.6)(38.3,24)
\Line(44.3,35.6)(48.3,24)
\Line(38.3,24)(32,10)
\Line(45,47.6)(73.3,47.6)
\Line(45,47.6)(44.3,35.6)
\Line(17,23.6)(8.3,10)
\Line(76.3,24.3)(77.3,10.6)
\Line(76.3,24.3)(69.3,10.3)
\Line(76,35.6)(76.3,24.3)
\Line(73.3,47.6)(91.3,35.6)
\Line(73.3,47.6)(76,35.6)
\Line(76,35.6)(86,24.3)
\Line(108.3,24.3)(125.3,10.3)
\Line(108.3,24.3)(117.3,10.3)
\Line(96,24.3)(109.3,10.3)
\Line(96,24.3)(101.3,10.3)
\Line(86,24.3)(85.3,10.3)
\Line(108.3,24.3)(117.3,10.3)
\Line(86,24.3)(93.3,10.3)
\Line(91.3,35.6)(108.3,24.3)
\Line(91.3,35.6)(96,24.3)
\Line(149.3,22.6)(139.3,10.3)
\Line(160.3,22)(152.7,10.3)
\Line(160.3,22)(159.3,10.3)
\Line(163.7,35.6)(160.3,22)
\Line(175,22.6)(172.7,10.3)
\Line(175,22.6)(166,10.3)
\Line(149.3,22.6)(146,10.3)
\Line(182.3,35.6)(175,22.6)
\Line(163.7,35.6)(149.3,22.6)
\Line(182.3,46.3)(163.7,35.6)
\Line(229.3,22.6)(219.3,10.3)
\Line(198.7,22.6)(206,10.3)
\Line(198.7,22.6)(212.7,10.3)
\Line(191.3,22.6)(199.3,10.3)
\Line(191.3,22.6)(192.7,10.3)
\Line(198.7,35.6)(198.7,22.6)
\Line(181.7,22.6)(186,10.3)
\Line(181.7,22.6)(179.3,10.3)
\Line(182.3,35.6)(181.7,22.6)
\Line(198.7,35.6)(191.3,22.6)
\Line(214.3,46.3)(229.3,35.6)
\Line(214.3,46.3)(198.7,35.6)
\Line(214.3,62.3)(214.3,46.3)
\Line(182.3,46.3)(182.3,35.6)
\Line(214.3,62.3)(248.3,46.3)
\Line(214.3,62.3)(182.3,46.3)
\Line(257,22.6)(266,10.3)
\Line(257,22.6)(259.3,10.3)
\Line(248.3,22.6)(252.7,10.3)
\Line(248.3,22.6)(246,10.3)
\Line(229.3,22.6)(226,10.3)
\Line(237.7,22.6)(239.3,10.3)
\Line(237.7,22.6)(232.7,10.3)
\Line(248.3,35.6)(248.3,22.6)
\Line(229.3,35.6)(237.7,22.6)
\Line(248.3,35.6)(257,22.6)
\Line(248.3,46.3)(248.3,35.6)
\Line(229.3,35.6)(229.3,22.6)
\Line(248.3,46.3)(269.3,35.6)
\Line(273.7,22.6)(279.3,10.3)
\Line(273.7,22.6)(272.7,10.3)
\Line(269.3,35.6)(273.7,22.6)
\Line(282.3,22.6)(292.7,10.3)
\Line(282.3,22.6)(286,10.3)
\Line(269.3,35.6)(282.3,22.6)
\put(210,-11){$T_{10}$}
\put(60.4,-11){$T_9$}
\put(29.7,35.6){\circle*{4}}
\put(48.3,24){\circle*{4}}
\put(44.3,35.6){\circle*{4}}
\put(45,47.6){\circle*{4}}
\put(38.3,24){\circle*{4}}
\put(28.3,23.6){\circle*{4}}
\put(17,23.6){\circle*{4}}
\put(77.3,10.6){\circle*{4}}
\put(76.3,24.3){\circle*{4}}
\put(76,35.6){\circle*{4}}
\put(73.3,47.6){\circle*{4}}
\put(109.3,10.3){\circle*{4}}
\put(101.3,10.3){\circle*{4}}
\put(85.3,10.3){\circle*{4}}
\put(93.3,10.3){\circle*{4}}
\put(69.3,10.3){\circle*{4}}
\put(56,10){\circle*{4}}
\put(48.3,10){\circle*{4}}
\put(40,10){\circle*{4}}
\put(24,10){\circle*{4}}
\put(16,10){\circle*{4}}
\put(0,10){\circle*{4}}
\put(108.3,24.3){\circle*{4}}
\put(91.3,35.6){\circle*{4}}
\put(96,24.3){\circle*{4}}
\put(86,24.3){\circle*{4}}
\put(8.3,10){\circle*{4}}
\put(32,10){\circle*{4}}
\put(125.3,10.3){\circle*{4}}
\put(117.3,10.3){\circle*{4}}
\put(117.3,10.3){\circle*{4}}
\put(139.3,10.3){\circle*{4}}
\put(152.7,10.3){\circle*{4}}
\put(159.3,10.3){\circle*{4}}
\put(160.3,22){\circle*{4}}
\put(172.7,10.3){\circle*{4}}
\put(166,10.3){\circle*{4}}
\put(146,10.3){\circle*{4}}
\put(175,22.6){\circle*{4}}
\put(149.3,22.6){\circle*{4}}
\put(163.7,35.6){\circle*{4}}
\put(219.3,10.3){\circle*{4}}
\put(206,10.3){\circle*{4}}
\put(212.7,10.3){\circle*{4}}
\put(199.3,10.3){\circle*{4}}
\put(192.7,10.3){\circle*{4}}
\put(198.7,22.6){\circle*{4}}
\put(186,10.3){\circle*{4}}
\put(179.3,10.3){\circle*{4}}
\put(181.7,22.6){\circle*{4}}
\put(191.3,22.6){\circle*{4}}
\put(198.7,35.6){\circle*{4}}
\put(214.3,46.3){\circle*{4}}
\put(182.3,35.6){\circle*{4}}
\put(182.3,46.3){\circle*{4}}
\put(214.3,62.3){\circle*{4}}
\put(266,10.3){\circle*{4}}
\put(259.3,10.3){\circle*{4}}
\put(252.7,10.3){\circle*{4}}
\put(246,10.3){\circle*{4}}
\put(226,10.3){\circle*{4}}
\put(239.3,10.3){\circle*{4}}
\put(232.7,10.3){\circle*{4}}
\put(248.3,22.6){\circle*{4}}
\put(237.7,22.6){\circle*{4}}
\put(257,22.6){\circle*{4}}
\put(248.3,35.6){\circle*{4}}
\put(229.3,22.6){\circle*{4}}
\put(229.3,35.6){\circle*{4}}
\put(248.3,46.3){\circle*{4}}
\put(279.3,10.3){\circle*{4}}
\put(272.7,10.3){\circle*{4}}
\put(273.7,22.6){\circle*{4}}
\put(286,10.3){\circle*{4}}
\put(282.3,22.6){\circle*{4}}
\put(269.3,35.6){\circle*{4}}
\put(292.7,10.3){\circle*{4}}
\end{picture}
\end{center}
\caption{Two examples of $T_k$.}%\label{1}
		\end{figure}
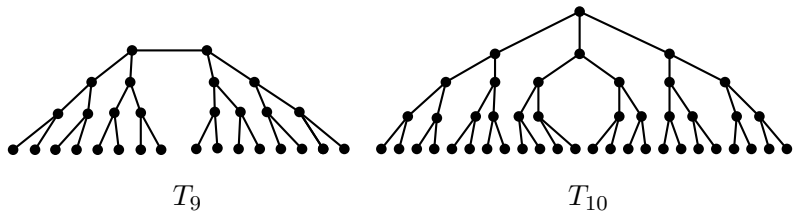
We first restate the definition of $T_k$ \cite{KT86} as follows. Suppose that $T_k$ is a tree with $\lfloor\frac{k}{2}\rfloor$ layers such that all vertices in each layer, except for the $\lfloor\frac{k}{2}\rfloor$-layer, have degree $3$ and the $1$-layer contains $k+1-2\lfloor\frac{k}{2}\rfloor$ vertices.  Two examples are  shown in Figure 1. Let $a_k=|T_k|$. Then $a_k=3\cdot 2^{t-1}-2$ if $k=2t$, $4\cdot 2^{t-1}-2$ if $k=2t+1$.  %还没定义完，明天继续.

 Let $T^0_k$ be a tree containing $\lceil\tfrac{k-2}{2}\rceil$ layers such that the $1$-layer has $\phi(k)$ vertices and then each vertex of the $i$-layer has degree $3$ for $ i\le \lceil\tfrac{k-2}{2}\rceil-2$, each vertex of the $(\lceil\tfrac{k-2}{2}\rceil-1)$-layer has degree $2$, where $\phi(k)=2$ for even $k$, $1$ otherwise. Two examples are presented in Figure 2. %\ref{2}
Evidently, $diam(T^0_k)=k-3$. Let $a^0_k=|T^0_k|$, then $a^0_k=3\cdot 2^{t-2}-2$ if $k=2t$, $9\cdot2^{t-3}-2$ if $k=2t+1$.

%Let $T^0_k$ be a tree containing $t-1$-layers   such that the $1$-layer contains $\phi(k)$ vertices and then each vertex of the $i$-layer has degree three for $ i\le t-3$, each vertex of the $(t-2)$-layer has degree two if $k=2t$, with $t$-layers for which $1$-layer contains a unique vertex and then each vertex of the $i$-layer has degree three for $i\le t-2$, each vertex of the $(t-1)$-layer has degree two if $k=2t+1$.  Evidently, $diam(T^0_k)=k-3$. Let $a^0_k=|T^0_k|$, then $a^0_k=3\cdot 2^{t-2}-2$ if $k=2t$, $9\cdot2^{t-3}-2$ if $k=2t+1$.% %检查定义
\vspace{3mm}
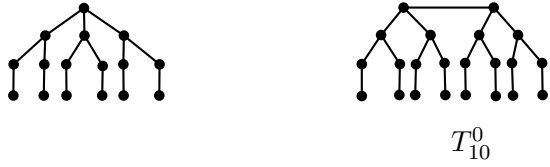
\begin{figure}[h]\label{2}
\begin{center}
\begin{picture}(200.2,54.3)\linethickness{0.8pt}
\Line(20,32.8)(20.1,21)
\Line(26.9,43.7)(20,32.8)
\Line(11.7,32.8)(11.5,21)
\Line(12.1,43.7)(11.7,32.8)
\Line(.1,32.8)(0,21)
\Line(12.1,43.7)(.1,32.8)
\Line(26.7,54.1)(12.1,43.7)
\Line(41.6,32.8)(41.7,21)
\Line(41.8,43.7)(41.6,32.8)
\Line(41.8,43.7)(55.3,32.8)
\Line(33.8,32.2)(33.6,21)
\Line(26.9,43.7)(33.8,32.2)
\Line(26.7,54.1)(41.8,43.7)
\Line(26.7,54.1)(26.9,43.7)
\Line(55.3,32.8)(55.2,21)
%\put(24.1,0){$T^0_9$}
\Line(146.2,32.9)(146,21.2)
\Line(139.1,43.9)(146.2,32.9)
\Line(131.8,32.9)(131.8,20.8)
\Line(139.1,43.9)(131.8,32.9)
\Line(147.6,54.3)(139.1,43.9)
\Line(152.3,32.8)(152,21.2)
\Line(157.8,43.9)(163.1,33.3)
\Line(157.8,43.9)(152.3,32.8)
\Line(147.6,54.3)(157.8,43.9)
\Line(147.6,54.3)(181.7,54.3)
\put(165.6,0){$T^0_{10}$}
\Line(163.1,33.3)(163.4,21.2)
\Line(182.3,33.3)(182.5,20.8)
\Line(170.9,33.3)(171.1,21.2)
\Line(176.2,43.9)(170.9,33.3)
\Line(176.2,43.9)(182.3,33.3)
\Line(181.7,54.3)(190.9,43.9)
\Line(181.7,54.3)(176.2,43.9)
\Line(188.6,33.3)(189,21)
\Line(190.9,43.9)(188.6,33.3)
\Line(190.9,43.9)(200,33.3)
\Line(200,33.3)(200.2,21.2)
\put(20,32.8){\circle*{4}}
\put(11.7,32.8){\circle*{4}}
\put(.1,32.8){\circle*{4}}
\put(12.1,43.7){\circle*{4}}
\put(41.7,21){\circle*{4}}
\put(41.6,32.8){\circle*{4}}
\put(33.6,21){\circle*{4}}
\put(20.1,21){\circle*{4}}
\put(33.8,32.2){\circle*{4}}
\put(11.5,21){\circle*{4}}
\put(0,21){\circle*{4}}
\put(41.8,43.7){\circle*{4}}
\put(26.9,43.7){\circle*{4}}
\put(26.7,54.1){\circle*{4}}
\put(55.2,21){\circle*{4}}
\put(55.3,32.8){\circle*{4}}
\put(146,21.2){\circle*{4}}
\put(146.2,32.9){\circle*{4}}
\put(131.8,20.8){\circle*{4}}
\put(131.8,32.9){\circle*{4}}
\put(139.1,43.9){\circle*{4}}
\put(147.6,54.3){\circle*{4}}
\put(152,21.2){\circle*{4}}
\put(152.3,32.8){\circle*{4}}
\put(157.8,43.9){\circle*{4}}
\put(163.4,21.2){\circle*{4}}
\put(163.1,33.3){\circle*{4}}
\put(182.5,20.8){\circle*{4}}
\put(171.1,21.2){\circle*{4}}
\put(170.9,33.3){\circle*{4}}
\put(182.3,33.3){\circle*{4}}
\put(176.2,43.9){\circle*{4}}
\put(181.7,54.3){\circle*{4}}
\put(189,21){\circle*{4}}
\put(188.6,33.3){\circle*{4}}
\put(190.9,43.9){\circle*{4}}
\put(200,33.3){\circle*{4}}
\put(200.2,21.2){\circle*{4}}
\end{picture}
\end{center}
\vspace{-5mm}
\caption{Two examples of $T^0_k$.}%\label{2}
		\end{figure}
\begin{figure}[h]\label{3}
\begin{center}
\begin{picture}(245,62.8)\linethickness{0.8pt}
\Line(28.2,48.1)(27.5,36.1)
\Line(28.2,48.1)(47.2,48.1)
\Line(28.2,48.1)(12.8,36.1)
\Line(63.5,36.1)(63.2,24.1)
\Line(63.5,36.1)(78,24.1)
\Line(47.2,36.1)(47,24.1)
\Line(47.2,48.1)(47.2,36.1)
\Line(47.2,48.1)(63.5,36.1)
\Line(27.5,36.1)(36.7,24.1)
\Line(12.8,36.1)(11.5,24.1)
\Line(12.8,36.1)(.2,24.1)
\Line(27.5,36.1)(27.8,24.1)
\Line(78,24.1)(78.2,9)
\Line(.2,24.1)(0,9)
\Line(27.8,24.1)(27.7,9)
\Line(167,36.1)(153.5,24.1)
\Line(153.5,24.1)(153.5,9)
\Line(181.5,48.1)(167,36.1)
\Line(167,36.1)(166.8,24.1)
\Line(197.5,62.8)(213.8,48.1)
\Line(197.5,62.8)(197.5,48.1)
\Line(197.5,48.1)(205.3,36.1)
\Line(197.5,48.1)(188.5,36.1)
\Line(197.5,62.8)(181.5,48.1)
\Line(181.5,48.1)(181.3,36.1)
\Line(188.5,36.1)(188.7,24.1)
\Line(181.3,36.1)(181.2,24.1)
\Line(213.8,48.1)(230,36.1)
\Line(213.8,48.1)(213.8,36.1)
\Line(213.8,36.1)(213.8,24.1)
\Line(205.3,36.1)(205.3,24.1)
\Line(230,36.1)(245,23.8)
\Line(230,36.1)(230.5,24.1)
\Line(245,23.8)(245,9)
\put(192.4,1){$T^1_{10}$}
\put(36.3,0){$T^1_9$}
\put(28.2,48.1){\circle*{4}}
\put(47.2,36.1){\circle*{4}}
\put(63.5,36.1){\circle*{4}}
\put(47.2,48.1){\circle*{4}}
\put(12.8,36.1){\circle*{4}}
\put(27.5,36.1){\circle*{4}}
\put(63.2,24.1){\circle*{4}}
\put(78,24.1){\circle*{4}}
\put(47,24.1){\circle*{4}}
\put(36.7,24.1){\circle*{4}}
\put(11.5,24.1){\circle*{4}}
\put(.2,24.1){\circle*{4}}
\put(27.8,24.1){\circle*{4}}
\put(153.5,9){\circle*{4}}
\put(153.5,24.1){\circle*{4}}
\put(78.2,9){\circle*{4}}
\put(0,9){\circle*{4}}
\put(27.7,9){\circle*{4}}
\put(167,36.1){\circle*{4}}
\put(166.8,24.1){\circle*{4}}
\put(197.5,48.1){\circle*{4}}
\put(197.5,62.8){\circle*{4}}
\put(188.5,36.1){\circle*{4}}
\put(181.3,36.1){\circle*{4}}
\put(181.2,24.1){\circle*{4}}
\put(188.7,24.1){\circle*{4}}
\put(181.5,48.1){\circle*{4}}
\put(213.8,36.1){\circle*{4}}
\put(213.8,24.1){\circle*{4}}
\put(213.8,48.1){\circle*{4}}
\put(205.3,36.1){\circle*{4}}
\put(205.3,24.1){\circle*{4}}
\put(230,36.1){\circle*{4}}
\put(230.5,24.1){\circle*{4}}
\put(245,9){\circle*{4}}
\put(245,23.8){\circle*{4}}
\end{picture}
\end{center}
\vspace{-3mm}
\caption{Two examples of $T^1_k$.}%\label{3}
		\end{figure}
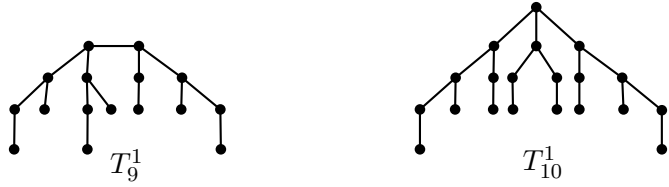

  We now  define a $T^1_k$ as follows. Suppose that $T^1_k$ is a tree with $\lfloor\frac{k}{2}\rfloor$ layers such that all vertices in $i$-layer have degree $3$ for $i\le\lfloor\frac{k}{2}\rfloor-3$, except for $\theta(k)$ vertices of degree three all vertices in $(\lfloor\frac{k}{2}\rfloor-2)$-layer have degree $2$,  except for $\theta(k)$ vertices of degree $2$ all vertices in $(\lfloor\frac{k}{2}\rfloor-1)$-layer have degree $1$ and the $1$-layer contains $k+1-2\lfloor\frac{k}{2}\rfloor$ vertices, where $\theta(k)=3$ for odd $k$, $2$ otherwise; furthermore, these $\theta(k)$ vertices of degree $3$ are adjacent to $\theta(k)$ vertices of degree $2$, and then  paths from these $\theta(k)$ vertices with degree $3$ to vertices in the $1$-layer are internal disjoint,  and each vertex of the $1$-layer possesses at least one of these paths. Two examples are presented in Figure 3.  Evidently, $diam(T^1_k)=k-2$. Let $a^1_k=|T^1_k|$, then $a^1_k=9\cdot 2^{t-4}+2$ if $k=2t$, $3\cdot2^{t-2}+4$ if $k=2t+1$.

% Let $T^1_k$ be a tree with $t$-layers such that when $k=2t+1$, the $1$-layer contains two vertices, each vertex of the $i$-layer has degree three for $i\le t-3$, each vertex of the $(t-2)$-layer has degree two except for vertices $v_{(t-2)1}$,$v_{(t-2)(j_{(t-2)}/2)}$ and $v_{(t-2)(j_{(t-2)})}$ which have degree three, and each vertex of the $(t-1)$-layer has degree one except for vertices $v_{(t-1)1}$,$v_{(t-1)(j_{(t-1)}/2-1)}$ and $v_{(t-1)(j_{(t-1)})}$ which have degree two, when $k=2t$, the $1$-layer contains a unique vertex, each vertex of the $i$-layer has degree three for $i\le t-3$, each vertex of the $(t-2)$-layer has degree two except for vertices $v_{(t-2)1}$ and $v_{(t-2)(j_{(t-2)})}$ which have degree three, and each vertex of the $(t-1)$-layer has degree one, except for vertices $v_{(t-1)1}$ and $v_{(t-1)(j_{(t-1)})}$ which have degree two.  {\color{red}Evidently, $diam(T^1_k)=k-2$. Let $a^1_k=|T^1_k|$, then {$a^1_k=9\cdot 2^{t-4}+2$ if $k=2t$, $3\cdot2^{t-2}+4$ if $k=2t+1$.}} %检查定义

We now recall that two known results of $P_k$-saturated graphs in \cite{KT86}.
\begin{lemma}\emph{(\cite{KT86})}\label{Pk}
  Let $T$ be a $P_k$-saturated tree. Then $T_k\subseteq T$.  
\end{lemma} 
%Note that $T_k$ is $P_k$-saturated.  
Lemma \ref{Pk} infers that  
 $T_k$ is the minimum $P_k$-saturated tree. Using the property, the following result is obtained.
\begin{theorem}\emph{(\cite{KT86})}
   If $n\ge a_k$ and $k\ge 6$, then $\sat(n,P_k)=n-\lfloor\tfrac{n}{a_k}\rfloor$.
\end{theorem}

It is unsurprising that $T_k$ is also $\{K_3,P_k\}$-saturated. The natural question is to determine the exact value of $\sat(n,\{K_3,P_k\})$. By means of  the tool ``layers'',  we can establish a statement of $\{K_3,P_k\}$-saturated trees that is analogous to Lemma \ref{Pk}.

\begin{lemma}\label{saturatedtree}
	For $k\ge10$, if $T$  is a $\{K_3,P_k\}$-saturated tree and not a star then $T_k^0\subseteq T$ or $T_k^1\subseteq T$, moreover, $e(T^0_k)> e(T^1_k)$.
\end{lemma} 

\begin{lemma}\label{pksaturated}
    $T^0_k$ and $T^1_k$ are $\{K_3,P_k\}$-saturated.
\end{lemma}
\proof    Let $T$ be a tree and $Y$ be the set of leaves in $T$. For convenience, let $L(w)$ (resp. $L'(w)$) be the longest (resp. shortest) path start with $w$ to all vertices of $Y$, and let $L_{w_0}(w)$ (resp. $L'_{w_0}(w)$)  be the longest (resp. shortest) path start with $w$ to all vertices of $Y$ forbidding the given vertex $w_0$.  We first show that $T^0_k$ is $\{K_3,P_k\}$-saturated. Note that $diam(T^0_k)=k-3$. For each edge $uv\not\in E(T^0_k)$, we consider $T^0_k+uv$. We assume without loss of generality that  $l(u)\ge l(v)$. Clearly, we can assume that $d_{T^0_k}(u,v)\ge 3$. %$d_{T^0_k}(u,v)\ge 3$. 

 {\bf Case 1} $k=2t+1$. We observe that $T^0_k$ has $t$-layers from its definition, in particular, the $1$-layer contains a unique vertex, say $v_{1}$. Observe that $u$ and $v$ belong to these paths from some leaves to $v_1$. If $u$ and $v$ belong to  the same such path, then $l(u)\ge l(v)+3$. Let $v'$ be a neighbor of the path with $ l(v')=l(v)+1$. Then we will find a path (say $P_{\sat}$) as $L_{v}(u)uvL(v)$ with length no less than $t-1+l(v)+t-1-l(v)+l(u)-l(v')\ge 2t$, where $L_{v}(u)$ goes through the vertex $v'$ with $l(v')=l(v)+1$. Hence, we assume that $u$ and $v$ lie on two different paths start with $v_1$. Let $w$ be the common vertex with maximum layer number of two paths start with $v_{1}$ and containing respectively $u$ and $v$.

Suppose $w=v_1$.  We now consider the case $l(u)> l(v)$. 
We observe that $T^0_k+uv$ contains a path  $P_{\sat}$ through $uv$ as $L_{v}(u)uvL_{v'}(v)$, where $v'$ is a neighbor of $v$ with $l(v')=l(v)-1$. Note that $e(P_{\sat})=e(L_{v'}(v))+e(L_{v}(u))+1=t-1+l(u)-1+t-1-(l(v)-1)+1\ge 2t=k-1$. 

Assume that $l(u)=l(v)$. Evidently, $l(u)\ge 3$.
let $u'$ be the neighbor of $u$ with $l(u')=l(u)-1$, which  is distinguished with $v_1$.  Hence, $T^0_k+uv$ has a path  $P_{sat}$ through $uv$ as $L'_{u}(u')u'uvL_{u'}(v)$. 
We observe that $e(P_{\sat})=e(L'_{u}(u'))+e(L_{u'}(v))+2=t-1+l(v)-1+t-1-(l(u')-1)+2\ge 2t=k-1$.

We now suppose $w\neq v_1$. If $u$ and $v$ belong to two different layers, then $l(u)>l(v)$. 
Similarly, we also find a path $P_{\sat}=L(u)uvL'(v)$ with order at least $2t+1$ in $T^0_k+uv$. If $l(u)=l(v)$, then  $T^0_k+uv$ includes a path $P_{\sat}=L_{u}(u')u'uvL(v)$ with order at least $2t+1$, where $u'$ is a neighbor of $u$ with $l(u')=l(u)-1$.

 {\bf Case 2} $k=2t$.  We observe that $T^0_k$ has $(t-1)$-layers from its definition, in particular, the $1$-layer contains two  vertices, say $u_1$ and $v_{1}$. It is trivial 
 for $d_T(u,v)=2$, so assume  $d_T(u,v)\ge 3$.  Assume that $u$ and $v$ are lying on the same shortest path $P_{uv} $ start with $u_1$ or $v_1$, say $u_1$. So $l(u)\ge l(v)+3$. We find  a path $P_{sat}=L(v)vuv'L_{v}(u)$ with order at least $2t$ in $T^0_k+uv$. Hence, we assume that $u$ and $v$ are lying on two different shortest paths start with $u_1$ or $v_1$.   Let $w$ be the common vertex with maximum layer number of two shortest paths start with $u_{1}$ forbidding $v_{1}$ (resp. $v_{1}$ forbidding $u_1$) and containing respectively $u$ and $v$.  Observe that $w$ does not exist if the unique path connecting $u$ and $v$ goes through $u_1$ and $v_1$. 

We first consider the case that $w$ does not exist in $T^0_k$. Without loss of generality, assume that $d_{T^0_k}(u,u_1)<d_{T^0_k}(u,v_1)$.   $T^0_k+uv$ includes a path $P_{\sat}=L_{v_1}(u)uvL_{u_1}(v)$ with order at least $2t$.

We next consider the case that $w$  exists in $T^0_k$ and assume  without loss of generality that $d_{T^0_k}(u_1,w)< d_{T^0_k}(v_1,w)$.
If $l(u)=l(v)$, then we will deduce that there is a $P_{\sat}=L'_{u}(u')u'uvL(v)$ with order at least $2t$ in $T^0_k+uv$, where $u'$ is a neighbor of $u$ with $l(u')=l(u)-1$.

If $l(u)>l(v)$, then we will find a  $P_{\sat}=L(u)uvL_{w}(v)$ with order at least $2t$ in $T^0_k+uv$.

Combining the two cases above, we are done as required. 

We now prove that $T^1_k$ is $\{K_3,P_k\}$-saturated. Evidently, $diam(T^1_k)=k-2$. We  consider $T^1_k+uv$ for each $uv\neq E(T^1_k)$. In fact, we can assume that $d_{T^1_k}(u,v)\ge 3$ and $l(u)\ge l(v)$.

{\bf Case 1} $k=2t$.
Observe that $T^1_k$ has $t$ layers, in particular, the $1$-layer contains a unique vertex, say $v_1$.
Observe that $u$ and $v$ belong to these paths from some leaves to $v_1$. If $u$ and $v$ lie on the same such path, then $l(u)\ge l(v)+3$,  then we will find a path $P_{\sat}=L_{v}(u)uvL(v)$ with order at least $2t$. Hence, we assume that $u$ and $v$ lie on two different paths start with $v_1$. Let $w$ be the common vertex with maximum layer number of two paths start with $v_{1}$ and containing respectively $u$ and $v$.  

Provided that $w=v_1$, then we will get a path $P_{\sat}=L(u)uvL_{w}(v)$ with order at least $2t$ in $T^1_k+uv$.

If $w\neq v_1$, then $2\le l(w)\le t-2$.  We thus find a path $P_{\sat}=L(u)uvL_{w}(v)$ in $T^1_k+uv$ for  $l(w)\le t-4$. We now consider the special case $l(w)\ge t-3$. Hence, $t-3\le l(w)\le t-2$. When $l(w)=t-3$, we will find a path $P_{\sat}=v'vuL(u)$, where $v'$ is a neighbor of $v$ distinguished with $w$ if $w$ is also a neighbor of $v$.
%either  $P_{\sat}=L'(v)vuL(u)$ with $l(u)>l(v)$ or $P_{\sat}=v'vvuL(u)$ with $l(u)=l(v)$, where $v'$ is the neighbor of $v$ with $l(v')=l(v)-1$.  
When $l(w)=t-2$, we deduce that $l(u)=t$ and $l(v)=t-1$. Hence, $T^1_k+uv$ contains a path $P_{\sat}=vuL(u)$. By direct calculation, all paths $P_{\sat}$ above have order at least $2t$.

{\bf Case 2} $k=2t+1$.

Note that $T^1_k$ has $t$-layers and  the first layer contains exactly two vertices, say $u_1$ and $v_{1}$.   We first assume that $u$ and $v$ are belonging to the same shortest path $P_{uv} $ start with $u_1$ or $v_1$, say $u_1$. Obviously, $l(u)\ge l(v)+3$. There is  a path $P_{sat}=L(v)vuL_{v}(u)$ with order at least $2t+1$ in $T^1_k+uv$. Hence, we assume that $u$ and $v$ are lying on two different shortest paths start with $u_1$ or $v_1$.   Let $w$ be the common vertex with maximum layer number of two shortest paths start with $u_{1}$ forbidding $v_{1}$ (resp. $v_{1}$ forbidding $u_1$) and containing respectively $u$ and $v$. We get an  observation that $w$ does not exist if the unique path connecting $u$ and $v$ goes through $u_1$ and $v_1$.   

We first consider the case that $w$ does not exist. It follows that
$u$ and $v$ are connected by a unique path going through $u_1$ and $v_1$. We assume without loss of generality that $d_{T^1_k}(u,u_1)<d_{T^1_k}(u,v_1)$ and $d_{T^1_k}(v,v_1)<d_{T^1_k}(v,u_1)$. Hence we deduce that  $T^1_k+uv$ contains a path $P_{\sat}=L_{v_1}(u)uvL_{u_1}(v)$ of order at least $2t+1$.

We thus assume that $w$ exists. Without loss of generality, assume that $d_{T^1_k}(w,u_1)<d_{T^1_k}(w,v_1)$. Clearly, $1\le l(w)\le t-2$. We now consider the case $1\le l(w)\le t-4$. Note that  $T^1_k+uv$ contains a path $P_{\sat}=L(u) uvL_{w}(v)$. If 
$l(w)= t-3$, then $T^1_k+uv$ has  a path $P_{\sat}=L(u)uvv')$, where $v'$ is a neighbor of $v$ distinguished with $w$ if $w$ is also a neighbor of $v$. If $l(w)=t-2$, then we deduce that $l(u)=t$ and $l(v)=t-1$. Hence, $T^1_k+uv$ contains a path $P_{\sat}=vuL(u)$. In conclude,  all paths $P_{\sat}$ have order at least $k$ by direct calculation.

%If $w=u_1$, then $T^1_k+uv$ contains a path $P_{\sat}=L(u) uvL_{w}(v)$. Evidently, the order of $P_{\sat}$ is no less than $2t+1$. Assume now that $w\neq u_1$. So $2\le l(w)\le t-2$. There is a path $P_{\sat}=L(u)uvL_{w}(v)$ for $l(w)\le t-4$. If $l(w)=t-3$, then $T^1_k+uv$ contains a path either  $P_{\sat}=L'(v)vuL(u)$ with $l(u)>l(v)$ or $P_{\sat}=v'vvuL(u)$ with $l(u)=l(v)$, where $v'$ is the neighbor of $v$ with $l(v')=l(v)-1$.  If $l(w)=t-2$, then we deduce that $l(u)=t$ and $l(v)=t-1$. Hence, $T^1_k+uv$ contains a path $P_{\sat}=vuL(u)$. In conclude,  all paths $P_{\sat}$ have order at least $k$ by direct calculation. 

 Together Case 1 with Case 2, we deduce that $T^1_k$ is $\{K_3,P_k\}$-saturated. \qed

%\noindent {\bf Remark} $e(T^0_k)= e(T^1_k)$ for $k\in\{6,9\}$ and $e(T^0_k)< e(T^1_k)$  for $k\in \{7,8\}$.
%{\color{blue} for $k\in \{4,5,6,7,8,9\}$, we will list the result in section ``Conclusion"}

%\begin{theorem}\label{saturationnumber2}
%If $n\ge a^1_k$ and {\color{red}$k\ge 10$}, then $\sat( n,\{P_k,K_3\}) =n-\lfloor n/a^1_k \rfloor$.
%\end{theorem}

%By Lemma \ref{saturatedtree} and Theorem \ref{saturationnumber2}, we obtain an upper bound and a lower bound of $\sat( n,K_3\cup P_k )$.
%\begin{theorem}\label{k3vpk}
	%For {\color{red}$k\ge 10$} and $n$ sufficiently large, we have that 
    %\begin{equation}\label{twobounds}
    %2+\sat( n,\{ K_3,P_k \} )\le \sat( n,K_3\cup P_k ) \le 6+\sat( n,\{ K_3,P_k \} ).
    %\end{equation}
%\end{theorem}

%\begin{theorem}\label{saturationjoin}
	%Let $G$ be a minimum $K_1\lor F$-saturated graph, then $e(G)=(n-1)+\sat(n-1,F)$ for sufficiently large $n$ and $K_1\lor H$ is a minimum $K_1\lor F$-saturated graph, where $H$ is a minimum $F$-saturated graph.
%\end{theorem}

\section{The proofs of Lemma \ref{saturatedtree}, Theorem  \ref{saturationnumber2} and Theorem \ref{k3vpk} }
In this section, we first prove Lemma \ref{saturatedtree}. And then by using the property of the minimum $\{ K_3,P_k\}$-saturated tree, we  show Theorem  \ref{saturationnumber2} and Theorem \ref{k3vpk}. %没有引用

  For convenience, we introduce some notation. Let $T$ be a tree with $diam(T)=s\ge3$. Then $T$ has a longest path $P_{s+1}$, set $P_{s+1}=v_{\lceil\frac{s+1}{2}\rceil}\cdots v_{21}v_{11} v'_{11}v'_{21}\cdots uv'_{\lceil\frac{s+1}{2}\rceil}$, in particular, $v_{11}$ is identified with  $v'_{11}$ for odd $s+1$. Let $P^l$ (resp. $P^{l'}$) be the unique shortest path  start with  $v_{11}$ forbidding $v'_{11}$ (resp. start with  $v'_{11}$ forbidding $v_{11}$)   end with some leaf of $T$ with order $l$ (resp. $l'$). Recall that  each vertex of $T$ can be divided into  $\lceil\frac{s+1}{2}\rceil$ layers according to the distance from it to $v_{11}$ or $v'_{11}$, moreover, it is lying on some path $P^l$ (or $P^{l'}$). In addition, let  %P_r=u_1u_2\cdots u_{r}$ and
  $P_{r_1}=u_1u_2\cdots u_{r_1}$ and $P_{r_2}=w_1w_2\cdots w_{r_2}$ be two paths. We call $P_{r_2}$ is a \emph{root-path} of $P_{r_1}$ at vertex $u_{i}$ if the two paths are only intersected at $w_1$ and some $u_{i}$.
We remark that  if we are to use the two types paths $P^{l}$ and $P^{l'}$ (If they exist simultaneously.)  to discuss the structure of a tree, then by symmetry, it suffices to use $P^{l}$ alone.

 \vspace{1mm}
 
\textbf{Proof of Lemma \ref{saturatedtree}:}
Let $T$ be a $\{K_3,P_k\}$-saturated tree and not a star and $ diam(T)=s$. Evidently, $3\leq s\leq k-2$. Let $P_{s+1}=v_1v_2 \cdots  v_{s+1}$  be a longest path of $T$. %需要统一符号 % $s\ge3$, otherwise, there is no $v_4$. Hence, $k\ge 7$. 
Hence, all vertices of $T$ can be partitioned into $\lceil\frac{ s+1}{2}\rceil$ layers such that the middle two vertices (or a unique vertex) of $P_{s+1}$ will belong to the $1$-layer.
We first verify the fact $k-3\le s\le k-2$ for $k\ge 5$. We observe that it holds trivially for $k=5,6$. Hence, we next assume $k\ge 7$. %% 下面的证明过程需要保证$t-1\ge 4$. 因此 $k\ge 10$
	We  assume to the contrary that $s\le k-4$. 
	 Since $T$ is $\{K_3,P_k\}$-saturated, $T+v_1v_4$ contain a copy of $P_k$, denoted by $P'_k$, and $v_1v_4\in E(P'_k)$. It follows that $P_{s+1}$ contains either a root-path %%在前面给出 定义 of a root-path
	  start with $v_2$ with length at least $k-1-(s-2)$ or a root-path  start with $v_2$ with length at least $k-1-(s-1)$ in $T$, where they both are different with the subpath $v_4\cdots v_{s+1}$. For the two cases, we thus obtain a path in $T$ with length at least $k-2$, a contradiction.  Based on the claim, we will take two cases to show our conclusion for $k\ge 10$.

	  %{\color{blue} Let $T$ be a $\{P_k,K_3\}$-saturated tree. Evidently, $diam(T)\leq k-2$. We first verify the fact $k-3\le diam(T)\le k-2$. We to the contrary assume that $diam(T)=s\le k-4$, then $T$ contains a longest path $P_{s+1}=v_1v_2 \ldots v_{s+1}$. Note that $T+v_1v_{s+1}$ contains a path with length $k$ as a copy of $P_k$ by our assumption. Hence, there are two vertices on $P_{s+1}$ as $x_i$ and $x_j$ such that there two subpaths respectively start with $x_i$ and $x_j$ and having }\qed

{\bf Case 1} 
$s=k-3$. % The descriptio of T can be partitioned into $\frac{diam(T)}{2}$, and root-path and definition of $P^l$ should be added before the proof of Lemma 2.1

%% For the case, we should show all vetitces with layer number $i\le t-3$ lie on at least two paths $P^{t-1}$ (or $P^{{t-1}'}$). In fact, there are some vertices with layer number $i= t-2$ maybe have degree two. 

We  show $T^0_k\subseteq T$ by the parity of $k$.  Observe that it suffices to show that each vertex belonging to the $i$-layer has degree at least three  for $i\in [1,\lceil\frac{ s+1}{2}\rceil-3]$ and degree at least two for $i=\lceil\frac{ s+1}{2}\rceil-2$.

{\bf Subcase 1.1 }
$k=2t$.

Observe that $s+1=k-2=2t-2$ for the subcase. For  notational  convenience, we  relabel  all vertices of  $P_{s+1}$ by using symmetric subscripts as $v_{(t-1)1}\cdots v_{21}v_{11} v'_{11} v'_{21}\cdots v'_{(t-1)1}$. 
%Formally, all vertices of $T$ can be partitioned into $t-1$ layers such that $v_{11}$ and $v'_{11}$ belong to the $1$-layer and  each vertex in the $i$-layer has the shortest path to $v_{11}$ or $v'_{11}$ with length $i-1$.
Observe that all vertices of $T$ can be partitioned into $t-1$ layers such that $v_{11}$ and $v'_{11}$ belong to the $1$-layer. 
In order to show $T^0_k\subseteq T$, it is sufficient to verify 
that for $i\in[1,t-3]$ each vertex of the $i$-layer is as the  common vertex with maximum layer number of at least two paths  having length $t-2$ start with $v_{11}$ (or $v'_{11}$) and end with some leaves.  %%这一段需要重写，
%Let $P^{l}$ (resp. $P^{l'}$)be a shortest path of $T$  start with $v_{11}$  (resp. $v'_{11}$) excluding $v'_{11}$ (resp. $v_{11}$) end with some leaf. Clearly, $l\le t-1$.  In the remaining part when we finish the proof through $P^{l}$ and $P^{l'}$, we always omit the case of $P^{l'}$ by the symmetry. 
Clearly,  $|P^{l}|$ and $|P^{l'}|$ are no more than $t-1$.
\begin{claim}\label{t-2a}
Suppose $P^{l}$ or $P^{l'}$ is a path of $T$. Then $l,l'\ge t-1$.
\end{claim}

\proof We assume to the contrary that $T$ contains a path $P^{l}=w_{11}(=v_{11})w_{21}\cdots w_{l1}$ with $d(w_{l1})=1$ and $l\le t-2$ which is not a subpath of $P_{s+1}$.  Let $i_0$ be the maximal subscript such that $w_{i_01}\in V(P_{s+1})$. We now consider $T+w_{(i_0+1)1}v_{(i_0+2)1}$. Note that it contains a copy of $P_k$ or $K_3$  by our assumption.  In fact, there is no triangle. Hence, $T+u_{(i_0+1)1}v_{(i_0+2)1}$ contains a copy of $P_k$. Observe that the all longest possible paths through $w_{(i_0+1)1}v_{(i_0+2)1}$ are $w_{l1}\cdots w_{(i_0+1)1}v_{(i_0+2)1}v_{(i_0+1)1}\cdots v_{11}\cdots v'_{11}$ and $v_{t1}\cdots v_{(i_0+2)1}u_{(i_0+1)1}u_{(i_0)1}\cdots v_{11}\cdots v'_{11}$.
Evidently, the lengths of these two paths are less than $k-1$, it follows that $T+u_{(i_0+1)1}v_{(i_0+2)1}$ does not contain a copy of $P_k$, a contradiction. 
Therefore, we complete the proof by the symmetry of $P^{l}$ and $P^{l'}$. \qed

By Claim \ref{t-2a}, we deduce that each path $P^{l}$ (or $P^{l'}$) has order $t-1$. We next show that  every path $P^{l}$ (or $P^{l'}$) has a root-path with length $t-2-(i-1)$ at  some vertex in the $i$-layer for $i\in[1,t-3]$.

 \begin{claim}\label{t-2b}
Each path $P^{l}$ \emph{(or  $P^{l'}$)} has a root-path at $u_{i1}$ with length $t-i-1$ for $i\in [1,t-3]$.
\end{claim}
 \proof  Let $j=t-i$ with $i\in [1,t-1]$. During the process of the proof, we first consider $T+w_{i1}w_{(i-3)1}$ for $i\in [4,t-1]$.  Take $j=1$, by our  assumption and $diam(T)=s$, $T+w_{(t-1)1}w_{(t-4)1}$ contains a copy of $P_k$, it follows that
 $P^{l}$ has a root-path at $u_{(t-1)1}$ with length $1$ or $0$ and a root-path at $u_{(t-2)1}$ with length $2$. We next take $j=2$, then $T+w_{(t-2)1}u_{(t-5)1}$ also contains a copy of $P_k$, which infers that 
 $P^l$ has  a root-path at $u_{(t-4)1}$ with length $3$. Based on these, by induction on $j$ we can show that  $P^l$ has a root-path at $w_{(t-j-2)1}(=w_{(i-2)1})$ with length $t-i+1$ for $3\le j\le t-4$. 
 
 Secondly, take $j=t-3$, we obtain that $P^{l}$ has  a root-path at $w_{11}$ with length $t-2$ by considering $T+w_{31}v'_{11}$. By the symmetry of $v_{11}$ and  $v'_{11}$, we can deduce that the above property of  $P^{l}$ is also valid for $P^{l'}$. Consequently, the conclusion is true. \qed

Combining Claims \ref{t-2a} and \ref{t-2b}, we thus deduce  $T^0_k\subseteq T$.

{\bf Subcase 1.2}
$k=2t+1$. %由于此时$T$的直径为s,那么$s+1=k-2=2t+1-2=2t-1$. 因此T可以分解成t层。所以子情形中的两个断言不能直接平移过来，需要处理一下。这里需要调整下

 In the subcase,  we  relabel  all vertices of  $P_{s+1}$ as $v_{t1}\cdots v_{21}v_{11} (=v'_{11}) v'_{21}\cdots v'_{t1}$  by using symmetric subscripts. %Note that $v_{11}=v'_{11}$ 
 Observe that all vertices of $T$ can be partitioned into $t$ layers such that $v_{11}$ belongs to the $1$-layer.     Using the same argument of Subcase 1.1, we can obtain that  each $P^{l}$ has order $t$ and  has the root-path property as required. We thus conclude that $T^0_k$ is a  subtree of $T$.

{\bf Case 2}
$s=k-2$.

 Observe that all vertices of $T$ can be partitioned  into $\lceil\tfrac{s+1}{2}\rceil$ layers. According to the definition of $T^1_k$,  we  verify that $T^1_k\subseteq T$ by the parity of $k$. 

{\bf Subcase 2.1}
$k=2t+1$.

For  notational  convenience, we  label  all vertices of  $P_{s+1}$ as $v_{t1}\cdots v_{21}v_{11} v'_{11} v'_{21}\cdots v'_{t1}$  by using symmetric subscripts. Recall that all vertices of $T$ can be  partitioned into $t$ layers such that $v_{11}$ and $v'_{11}$ belong to the $1$-layer and  each vertex in the $i$-layer has the shortest path to $v_{11}$ or $v'_{11}$ with length $i-1$. In order to show $T^1_k\subseteq T$, it is sufficient to verify 
that for $i\in[1,t-3]$ each vertex of the $i$-layer is lying on at least two paths start with $v_{11}$ or $v'_{11}$ and end with some leaves having length $t-2$ or $t-3$ such that it is the common vertex of these two paths  with maximum layer number, moreover, in all these paths,  at least  three paths have length $t-2$ with maximum layer number $1$. %%表述不准确
%Let $P^{l}$ (resp. $P^{l'}$)be a shortest path of $T$  start with $v_{11}$  (resp. $v'_{11}$) excluding $v'_{11}$ (resp. $v_{11}$) end with some leaf. Clearly, $l\le t$.  In the remaining part when we finish the proof through $P^{l}$ and $P^{l'}$, we always omit the case of $P^{l'}$ by the symmetry. 
Clearly, the order $l$ of $P^{l}$ and $P^{l'}$ is no more than $t$.
\begin{claim}\label{t-3}
Suppose $P^{l}$ (or $P^{l'}$) is a path in $T$ as defined above. Then $l,l'\ge t-1$. 
\end{claim}

\proof We prove this claim by contradiction. Assume that there is a path $P^{l}=u_{11}(=v_{11})u_{21}\cdots u_{l1}$ with $d(u_{l1})=1$ and $l\le t-2$ such that it is not a subpath of $P_{s+1}$. For convenience, let $i_0$ be the maximal subscript such that $u_{i_01}$ is also lying on the path $P_{s+1}$. We now consider $T+u_{(i_0+1)1}v_{(i_0+2)1}$. Note that it contains a copy of $P_k$ or $K_3$. Evidently, it does not contain a triangle. So $T+u_{(i_0+1)1}v_{(i_0+2)1}$ includes a copy of $P_k$. But the two longest paths through $u_{(i_0+1)1}v_{(i_0+2)1}$ are $u_{l1}\cdots u_{(i_0+1)1}v_{(i_0+2)1}v_{(i_0+1)1}\cdots v_{11}\cdots v'_{t1}$ and $v_{t1}\cdots v_{(i_0+2)1}u_{(i_0+1)1}u_{i_01}\cdots v_{11}\cdots v'_{t1}$.
Clearly, their lengths are no more than $k-2$, a contradiction.
By the above argument and the symmetry of $P^{l}$ and $P^{l'}$, we complete the proof. \qed
  
From above Claim, we get $l,l'\in\{t-1,t\}$. %We next show  that   each vertex in the $i$-layer lying on every path $P^l$ (or  $P^{l'}$) has degree at least three for $i\in [1,t-3]$ and at least two for $i=t-2$.  
Recall that  $P^{l}=u_{11}(=v_{11})u_{21}\cdots u_{l1}$.
Since the vertex $u_{(t-2)1}$ of a $P^{t}$ maybe has a leaf neighbor, say $u'_{(t-2)1}$, in fact, $u_{(t-2)1}$ is also contained in  a $P^{t-1}$. Based on the reason, when we consider the case  $P^{t-1}$,  we always assume that  $u_{(t-2)1}$ does not belong to  some $P^{t}$ by the inclusion-exclusion principle. 
We next show that for each $i\in[1,t-2]$. $P^{t}$ has a root-path at $u_{i1}$ in the $i$-layer with length $t-i+1$ or $t-i$.

 \begin{claim}\label{t}
Let $l=t$. Then $P^{l}$ has a root-path at $u_{i1}$ with length $t-i+1$ or $t-i$ for $i\in [1,t-2]$.
\end{claim}
 \proof % We first show that for $i\in[1,t-1]$ each vertex of $v_{i1}$ and $v'_{i1}$ is contained in another path start with  $v_{11}$ and $v'_{11}$  different with the subpath of $P_{s+1}$ with length $t-2$ or $t-3$.
 
 Let $j=t-i+1$ with $i\in [1,t]$. Similar to the argument of subcase 1.1,  we first consider $T+u_{t1}u_{(t-3)1}$ with $j=1$. By our  assumption and $diam(T)=s$, we verify that 
 $P^t$ has a root-path at $u_{(t-1)1}$ with length $1$ or $0$ and a root-path at $u_{(t-2)1}$ with length $2$ or $1$. For $j=2$, we consider $T+u_{(t-1)1}u_{(t-4)1}$ and obtain that 
 $P^t$ has  a root-path at $u_{(t-3)1}$ with length $3$ or $2$. Based on these, by induction on $j$ we can show that  $P^t$ has  a root-path at $u_{(t-j-1)1}(=u_{(i-2)1})$ with length $t-i+2$ or $t-i+1$ for $3\le j\le t-3$. 
 
 For $j=t-2$, we consider $T+u_{31}v'_{11}$ and obtain that $P^t$ has  a root-path at $u_{11}$ with length $t-1$ or $t-2$. For $j=t-1$, we deduce that $P^t$ has a root-path at $v'_{11}$ with length $t-1$. Together case $j=t-2$ and the symmetry of $v_{11}$ and  $v'_{11}$, we assume without loss of generality that there is a root-path at $u_{11}$ with length $t-1$, and there is a root-path at $v'_{11}$ with length $t-2$. \qed

\begin{claim}\label{t-1}
Let $l=t-1$. Then $P^{l}$ has a root-path at $u_{i1}$ with either length $t-i$ or $t-i-1$ for $i\in [1,t-3]$. %or  with length $1$ or $0$ for $i=t-2$.
\end{claim} %% For the speical case, the vertex $u_{(t-2)1}$ maybe have degree two. Hence, we should deal with it by what way?

\proof By our convention on $P^{t-1}$ and using the same argument of Claim \ref{t}, the conclusion is true. \qed

Combining Claims \ref{t} and \ref{t-1} and the symmetry of $P_{l}$ and $P'_{l}$,  we indeed deduce  $T^1_k\subseteq T$.

{\bf Subcase 2.2}
$k=2t$.

For  notational  convenience, we  label  all vertices of  $P_{s+1}$ as $v_{t1}\cdots v_{21}v_{11} (=v'_{11})  v'_{21}\cdots v'_{t1}$  by using symmetric subscripts. Note that in the case $v_{11}=v'_{11}$, Claims \ref{t-3}, \ref{t} and \ref{t-1} still valid.   Hence by using the same argument on  the path $P^l$, 
we can deduce that $T^1_k\subseteq T$.

In addition, by direct calculation, we obtain that $e(T^0_k)>e(T^1_k)$ for $k\ge 10$. Therefore, we complete the proof. %明天重新检查下。
\qed

By means of Lemma \ref{saturatedtree}, we let $ G_0=G_1\cup G_2\cup \cdots  \cup  G_t$ and $n\equiv n_0( mod\ a_{k}^{1} )$ such that $G_1$ is a  $\{K_3,P_k\}$-saturated tree with $|V(G_1)|=n_0+a_{k}^{1}$ and $G_i$ is a copy of $T^1_k$ for $i\in \{2,3,\ldots, t\}$. Next, we will show that $G_0$ is $\{K_3,P_k\}$-saturated.

\begin{lemma}\label{G0saturated}
	$G_0$ is $\{K_3,P_k\}$-saturated and $e(G_0)=n-\lfloor n/a'_k \rfloor$.
\end{lemma}
\begin{proof}
	From the construction of $G_0$, we can observe that $G_0$ is $\{K_3,P_k\}$-free. Hence, we next show that $G_0+uv$ contains a copy of $P_k$ or $K_3$ for each edge $uv\in  E( \overline{G_0})$. If $u$ and $v$ belong to one component of $G_0$, then it is true from Lemma \ref{pksaturated} and Lemma \ref{saturatedtree}. Hence, we assume that  $u$ and $v$ come from two components of $G_0$. Recall that the definition of $L(u)$, there is a path as $L(u)uvL(v)$ of order at least $k$. 

    Therefore, we are done.   
\end{proof}
\textbf{Proof of Theorem \ref{saturationnumber2}:}
 For convenience, we suppose that $G$ be a minimum  $\{K_3,P_k\}$-saturated graph. From Lemma \ref{G0saturated}, we have that $e(G)\le e(G_0)$. We thus need to show $e(G)\ge e(G_0)$. We assume to the contrary that  $e(G)< e(G_0)$. If $G$ is connected, then $e(G)\ge n-1\ge e(G_0)$, a contradiction. If $G$ is disconnected and its each component contains cycles, then $e(G)\ge n> e(G_0)$, a contradiction. Hence, we assume that $G$ contains at least one component that is a tree. Formally, let $G$ contain $s$ cycle components as $G_1,G_2,\ldots,G_s$ and $l$ tree components as $G_{s+1},G_{s+2},\ldots,G_{s+l}$ with $n_0=\sum^s_{i=1}|G_i|$ and $n-n_0=\sum^{s+l}_{i=s+1}|G_i|$. Evidently, $\lfloor \tfrac{n-n_0}{a^1_k} \rfloor\ge l$.  From Lemma \ref{saturatedtree}, $T^0_k\subseteq G_i$ or $T^1_k\subseteq G_i$ for $i\ge s+1$. 

 Hence, we deduce that 
 \begin{equation*}
   \begin{split}
      e(G)&\ge \sum^s_{i=1}|G_i|+\sum^{s+l}_{i=s+1}(|G_i|-1)  \\
      & =n_0+ n-n_0-l\\
      %&\ge  n_0+ n-n_0-(n-n_0)/a^1_k\\
      &\ge n_0++ n-n_0-\left\lfloor \tfrac{n-n_0}{a^1_k}\right\rfloor\\
      %&\ge n-n/a^1_k.
      &\ge n-\left\lfloor \tfrac{n}{a^1_k}\right\rfloor.
   \end{split} 
 \end{equation*} Therefore, we finish the proof. \qed

At the end of this section, we shall show two bounds of $sat(n, K_3\cup P_k )$. We first construct a $(K_3\cup P_k)$-saturated graph. Let $H_0=Q_1 \cup Q_2\cup \cdots \cup Q_m$ where $Q_1$ contains a copy of $K_4$, denoted by $Q'_1$ and each vertex of $Q'_1$ hang a copy of $T^1_k$ (see Figure 4), $Q_2$ be a $\{K_3,P_k\}$-saturated tree, and $Q_i$ be a copy of $T^1_k$ for $i\in \{3,4,\ldots,m\}$. For convenience, we let $V(Q'_1)=\{u_1,u_2,u_3,u_4\}$ and $Q_1\setminus E(Q'_1)=Q''_1\cup Q''_2 \cup Q''_3\cup Q''_4$, where $Q''_i$ is a copy of $T^1_k$ and contains the vertex $u_i$.

\vspace{3mm}
\begin{figure}[h]\label{4}		
\begin{center}
\begin{picture}(118.4,95.4)\linethickness{0.8pt}
\Line(.4,71.6)(.6,60.2)
\Line(8.5,83.8)(.4,71.6)
\Line(0,23.6)(.3,35.1)
\Line(17.9,95.3)(29.7,95.4)
\Line(17.9,95.3)(17.9,83.9)
\Line(18.1,71.7)(18.1,60.2)
\Line(17.9,83.9)(23.6,71.7)
\Line(17.9,83.9)(18.1,71.7)
\Line(8.5,83.8)(8.6,71.7)
\Line(17.9,95.3)(8.5,83.8)
\put(32.5,33.7){$u_3$}
\Line(8.6,11.5)(8,23.6)
\Line(8.6,11.5)(0,23.6)
\Line(17.3,0)(8.6,11.5)
\Line(17.6,11.5)(17.3,23.7)
\Line(17.3,0)(17.6,11.5)
\Line(29.1,.1)(29.1,11.6)
\Line(28.8,23.7)(28.8,35.2)
\Line(29.1,11.6)(23.3,23.7)
\Line(29.1,11.6)(28.8,23.7)
\Line(29.7,95.4)(38.4,83.9)
\Line(29.4,83.9)(29.6,71.7)
\Line(29.7,95.4)(29.4,83.9)
\put(32.5,60.6){$u_1$}
\Line(38.5,11.7)(38.3,23.8)
\Line(38.5,11.7)(46.5,23.9)
\Line(29.1,.1)(38.5,11.7)
\Line(38.4,83.9)(38.9,71.7)
\Line(38.4,83.9)(46.9,71.7)
\Line(29.1,.1)(17.3,0)
\Line(46.5,23.9)(46.3,35.3)
\Line(46.1,35.2)(71.7,60.2)
\Line(46.6,60.2)(72.1,35.2)
\Line(46.1,35.2)(72.1,35.2)
\Line(46.6,60.2)(46.1,35.2)
\Line(46.6,60.2)(71.7,60.2)
\Line(46.9,71.7)(46.6,60.2)
\put(76.2,34.2){$u_4$}
\put(76,60.9){$u_2$}
\Line(80.4,11.5)(71.8,23.6)
\Line(80.4,11.5)(79.8,23.6)
\Line(71.8,23.6)(72.1,35.1)
\Line(71.7,71.5)(71.9,60)
\Line(79.8,83.6)(79.9,71.5)
\Line(79.8,83.6)(71.7,71.5)
\Line(89.2,95.2)(79.8,83.6)
\Line(71.7,60.2)(72.1,35.2)
\Line(89.1,0)(80.4,11.5)
\Line(100.9,.1)(89.1,0)
\Line(89.1,0)(89.4,11.5)
\Line(100.9,11.6)(95.2,23.7)
\Line(89.4,11.5)(89.2,23.7)
\Line(89.2,95.2)(101.1,95.2)
\Line(89.2,95.2)(89.2,83.7)
\Line(89.4,71.5)(89.4,60)
\Line(89.2,83.7)(94.9,71.5)
\Line(89.2,83.7)(89.4,71.5)
\Line(100.9,.1)(110.3,11.7)
\Line(100.9,.1)(100.9,11.6)
\Line(100.9,11.6)(100.7,23.7)
\Line(100.7,23.7)(100.6,35.2)
\Line(101.1,95.2)(109.7,83.7)
\Line(100.7,83.7)(100.9,71.5)
\Line(101.1,95.2)(100.7,83.7)
\Line(110.3,11.7)(118.4,23.9)
\Line(110.3,11.7)(110.2,23.8)
\Line(109.7,83.7)(110.2,71.5)
\Line(109.7,83.7)(118.2,71.5)
\Line(118.4,23.9)(118.1,35.3)
\Line(118.2,71.5)(117.9,60)
\put(.6,60.2){\circle*{4}}
\put(.4,71.6){\circle*{4}}
\put(.3,35.1){\circle*{4}}
\put(18.1,60.2){\circle*{4}}
\put(18.1,71.7){\circle*{4}}
\put(17.9,83.9){\circle*{4}}
\put(8.6,71.7){\circle*{4}}
\put(8.5,83.8){\circle*{4}}
\put(17.9,95.3){\circle*{4}}
\put(8,23.6){\circle*{4}}
\put(0,23.6){\circle*{4}}
\put(17.3,23.7){\circle*{4}}
\put(17.6,11.5){\circle*{4}}
\put(8.6,11.5){\circle*{4}}
\put(28.8,35.2){\circle*{4}}
\put(28.8,23.7){\circle*{4}}
\put(23.3,23.7){\circle*{4}}
\put(29.1,11.6){\circle*{4}}
\put(29.6,71.7){\circle*{4}}
\put(29.4,83.9){\circle*{4}}
\put(29.7,95.4){\circle*{4}}
\put(23.6,71.7){\circle*{4}}
\put(38.3,23.8){\circle*{4}}
\put(38.5,11.7){\circle*{4}}
\put(38.9,71.7){\circle*{4}}
\put(38.4,83.9){\circle*{4}}
\put(17.3,0){\circle*{4}}
\put(46.3,35.3){\circle*{4}}
\put(46.5,23.9){\circle*{4}}
\put(29.1,.1){\circle*{4}}
\put(46.1,35.2){\circle*{4}}
\put(46.6,60.2){\circle*{4}}
\put(46.9,71.7){\circle*{4}}
\put(80.4,11.5){\circle*{4}}
\put(79.8,23.6){\circle*{4}}
\put(71.8,23.6){\circle*{4}}
\put(72.1,35.1){\circle*{4}}
\put(71.9,60){\circle*{4}}
\put(79.9,71.5){\circle*{4}}
\put(71.7,71.5){\circle*{4}}
\put(79.8,83.6){\circle*{4}}
\put(72.1,35.2){\circle*{4}}
\put(71.7,60.2){\circle*{4}}
\put(89.1,0){\circle*{4}}
\put(89.4,11.5){\circle*{4}}
\put(95.2,23.7){\circle*{4}}
\put(89.2,23.7){\circle*{4}}
\put(89.4,60){\circle*{4}}
\put(94.9,71.5){\circle*{4}}
\put(89.4,71.5){\circle*{4}}
\put(89.2,83.7){\circle*{4}}
\put(89.2,95.2){\circle*{4}}
\put(100.9,.1){\circle*{4}}
\put(100.9,11.6){\circle*{4}}
\put(100.7,23.7){\circle*{4}}
\put(100.6,35.2){\circle*{4}}
\put(100.9,71.5){\circle*{4}}
\put(100.7,83.7){\circle*{4}}
\put(101.1,95.2){\circle*{4}}
\put(110.3,11.7){\circle*{4}}
\put(110.2,23.8){\circle*{4}}
\put(110.2,71.5){\circle*{4}}
\put(109.7,83.7){\circle*{4}}
\put(118.1,35.3){\circle*{4}}
\put(118.4,23.9){\circle*{4}}
\put(117.9,60){\circle*{4}}
\put(118.2,71.5){\circle*{4}}
\end{picture}
\end{center}
\caption{An example of $Q_1$ for $k=9$.}%\label{4}
		\end{figure}
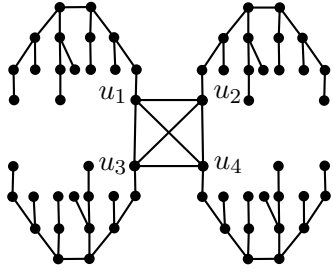

\begin{lemma}\label{H0}
	$H_0$ is  $( K_3\cup P_k )$-saturated graph and $e(H_0)=6+\sat( n,\{ K_3,P_k\} ) $.
\end{lemma}
\begin{proof}
		From the construction of $H_0$, we can observe that $H_0$ is $( K_3\cup P_k )$-free. For any edge $uv\notin E(H_0)$, we will show that $H_0+uv$ contains a copy of $K_3\cup P_k$, denoted by $K_{3}^{uv}$ and $P_{k}^{uv}$. Then, we can discuss it in two cases as follows.

        {\bf Case 1}
        $u$ and $v$ belong to the two different components.

    Observe that $uv\in E(P_{k}^{uv})$. Assume that one of 
 $u$ and $v$ is belonging to $V(Q_1)$, say $u\in V(Q_1)$. So we without loss of generality assume that $u\in V(Q''_{1})$ and $v\in V(Q_{i})$ with $i\ge 2$. It follows that $Q'_1-u$ forms $K_{3}^{uv}$ and $Q''_{1}\cup Q_{i}+uv$ contains $E(P_{k}^{uv})$. The remaining case is that $u$ and 
 $v$ do not belong to $V(Q_1)$. So there are two components of $H_0$, say $Q_i$ and $Q_j$ with $i\neq j$ and $i,j\ge 2$, such that $Q_i\cup Q_j+uv$ contains $E(P_{k}^{uv})$. In addition, $K_{3}^{uv}$ is a subgraph of $Q'_1$.

    {\bf Case 2}
    $u$ and $v$ belong to the same component.

    {\bf Subcase 2.1}
$u,v\in V(Q_i)$, for $i\in \{2,3,\ldots,m\}$.

    It is clear that $Q'_1$ contains $K_{3}^{uv}$ if $uv\in E(P_{k}^{uv})$. Hence, we assume $uv\in E(K_{3}^{uv})$.  It is not difficult to check that  $P_{k}^{uv}$ is contained in $Q_1$.

    {\bf Subcase 2.2}
    $u,v\in V(Q_1)$.

We first assume that $u$ and $v$ belong to the same copy of $T^1_k$.
 Without loss of generality, we suppose $u,v\in V(Q''_1)$. If $uv\in E(P_{k}^{uv})$, then $Q'_1-u_1$  forms $K_{3}^{uv}$. If $uv\in E(K_{3}^{uv})$, then $Q''_2\cup Q''_3+u_2u_3$. Obviously contains $P_{k}^{uv}$. We thus  assume that 
  $u\in V(Q''_i)$ and $u\in V(Q''_j)$ with $i\ne j$.
  If $uv\in E(K_{3}^{uv})$, then $Q_1-Q''_i- Q''_j$ contains  $P_{k}^{uv}$. For the case $uv\in E(P_{k}^{uv})$, we observe that $ Q''_i\cup Q''_j+uv$ creates $P_{k}^{uv}$, which infers that $P_{k}^{uv}$ goes through at most one of $u_i$ and $u_j$. Hence, we can assume $u_i\not\in V(P_{k}^{uv})$. Then $Q'_1-u_j$ forms $K_{3}^{uv}$.	
\end{proof}%\vspace{-2mm}
\textbf{Proof of Theorem \ref{k3vpk}:} For convenience, we suppose that $H$ be a minimum  $(K_3\cup P_k)$-saturated graph. From Lemma \ref{H0}, we have that $e(H)\le e(H_0)$. In the rest of the proof, it is sufficient to show that $e(H)\ge 2+\sat( n,\{ K_3,P_k \})$. 
If $H$ is $K_3$-free, then $H$ is $K_3$-saturated by the choice of $H$. In fact, $K_1\lor \overline{K_{n-1}}$ is a minimum $K_3$-saturated graph. It follows that $e(H)\ge n-1 \ge  2+\sat( n,\{ K_3,P_k \})$ holds by sufficiently large $n$. So we now assume that $H$ contains at least one triangle.
If $H$ is connected, then $e(H)\ge n-1\ge 2+\sat( n,\{ K_3,P_k\})$ holds by sufficiently large $n$. We claim that $H$ contains at least two tree components. Otherwise, $H$ contains at most one tree component, then $e(H)\ge n-1\ge 2+\sat( n,\{ K_3,P_k \} )$ holds by sufficiently large $n$. 
Thus, we conclude that $H$ contains a copy of $K_3$ and at least two tree components. We now claim that each tree component of $H$ is $\{ K_3,P_k \}$-saturated. 
Let $T$ be an arbitrary tree component of $H$. Obviously, $T$ is $K_3$-free. Moreover, we know that $T$ is $P_k$-free, otherwise, $H$ contains a copy of $K_3\cup P_k$, a contradiction. Since $H+uv$ contains a copy of $K_3\cup P_k$ for an arbitrary edge $uv\in E(\overline{T})$. We thus deduce that $T+uv$ creates a triangle or a  path with order at least $k$ for an arbitrary edge $uv\in E(\overline{T})$. Hence, the claim is true.

 For convenience, label $Q_{1}^{*},Q_{2}^{*},\ldots,Q_{s}^{*}$ as the $s$ components of $H$, where $Q_{1}^{*}$ contains a copy of $K_3$, denoted by $K$ and the last $l(\ge 2)$ components $Q_{s-l+1}^{*},\ldots,Q_{s}^{*}$ are trees of order at least $a^1_k$ by Lemma \ref{saturatedtree}. Set $n_0=|Q_{1}^{*}|$ and $n_{0}'=\sum_{i=2}^{s-l}{|Q_{i}^{*}|}$ for short.

	Note that if $H$ contains a $K_p$ with $p\ge4$, then we have
\begin{equation}
\begin{split}
e( H )& \ge e( Q_{1}^{*} ) +e\left( \sum_{i=2}^{s-l}{Q_{i}^{*}} \right) +e\left( \sum_{i=s-l+1}^s{Q_{i}^{*}} \right) \\
%&\ge \binom{p}{2} +n_0-p+n_{0}'+n-n_0-n_{0}'-(n-( n_0-n_{0}'))/{a_{k}^{1}
&\ge \binom{p}{2} +n_0-p+n_{0}'+n-n_0-n_{0}'-\left\lfloor \frac{n-( n_0-n_{0}' )}{a_{k}^{'}} \right\rfloor 
\\
%&(p^2-3p)/2+n-(n-( n_0-n_{0}'))/a_{k}^{1}
&\ge \frac{p^2-3p}{2}+n-\left\lfloor \frac{n-( n_0-n_{0}' )}{a_{k}^{1}} \right\rfloor 
\\
%&\ge 2+n-n/a_{k}^{1}.
&\ge 2+n-\left\lfloor \frac{n}{a_{k}^{1}} \right\rfloor.
\end{split}
	\end{equation} We thus are done. Hence, we assume that the maximum clique of $H$ is $K_3$. Let $V( K ) =\{ v_1,v_2,v_3 \}$. We consider $H+v_iw$ for $i \in \{ 1,2,3 \} $ and $w\in Q_{s}^{*}$.  Hence, $H+v_iw$ contains a copy of $K_3\cup P_k$, denoted by $K_{3}^{v_iw}\cup P_{k}^{v_iw}$. In fact, $v_iw\in E( P_{k}^{v_iw} ) $.\vspace{3mm}
 
{\bf Claim 2}  $e(Q_{1}^{*})\ge n_0+2$.

 \proof
Observe that the statement holds if $Q_{1}^{*}$ contains at least three different cycles. Observe that $Q_{1}^{*}$ contains at least two cycles $K$ and $K_{3}^{v_iw}$. Hence, we can assume that 
$Q^{*}_{1}$ contains exactly two distinguished triangles, where one is the $K$. It leads to that all these triangles, such as $K_{3}^{v_iw}$, coincide. Hence, $Q_{1}^{*}$ exactly contains two triangles one $K$ and the other $K_{3}^{v_iw}$ such that $V(K_{3}^{v_iw})\cap V(K)=\emptyset$. (If not, we are done.) For convenience, set
 $V(K_{3}^{v_iw}) =\{ v_4,v_5,v_6 \}$.

Next, we assume that there are at least two paths connecting $K_{3}^{v_iw}$ and $K$. Otherwise, 
there is exactly one path connecting $K_{3}^{v_iw}$ and $K$, denoted by $P_{\Delta}$.  Without loss of generality, assume that its two ends are vertices $v_1$ and $v_4$.
We now consider  $H+v_2v_5$. From our assumption,   it contains a copy of $K_3\cup P_k$, denoted by $K_{3}^{v_2v_5}\cup P_{k}^{v_2v_5}$,  and then $v_2v_5\in E( K_{3}^{v_2v_5} )$ (If not, then $v_2v_5\in E( P_{k}^{v_2v_5} )$.), hence, there exists the third triangle different with $K_{3}^{v_iw}$ and $K$, a contradiction.) It results that there are two paths  connecting $K_{3}^{v_iw}$ and $K$ one $P_{\Delta}$ and another going through $v_2$ and $v_5$. We also get a contradiction. Consequently, for the case we deduce that $e(Q_1^{*})\ge n_0+2$ as required.\qed

By Claim 2, we have
\begin{equation*}
		\begin{split}
e( H )&\ge e( Q_{1}^{*} ) +e\left( \sum_{i=2}^{s-l}{Q_{i}^{*}} \right) +e\left( \sum_{i=s-l+1}^s{Q_{i}^{*}} \right) 
\\&
\ge n_0+2+n_{0}'+n-n_0-n_{0}'-\left\lfloor \frac{n-( n_0-n_{0}') }{a_{k}^{1}} \right\rfloor 
\\&
\ge 2+n-\left\lfloor \frac{n-( n_0-n_{0}' )}{a_{k}^{1}} \right\rfloor 
\\&
\ge 2+n-\left\lfloor \frac{n}{a_{k}^{1}} 
\right\rfloor.
\end{split}
\end{equation*}
 We thus are done.\qed

\section{The proof of Theorem \ref{saturationjoin}}
In this section, we will research saturation number of a join of two graphs. More accurately,  
we show that saturation number of the join of an isolated vertex and a linear forest. Recall that we use $K_1\lor F$ to denote the join, where $F$ is a linear forest without isolated vertices.  Before presenting the proof of Theorem \ref{saturationjoin}, we need some preliminary conclusions. 

Bollob\'{a}s \cite{B78} proposed the following  lemma on the minimum size of $2$-connected graphs, which will play a key role in the proof of Theorem \ref{saturationjoin}.

%In 1986, Tuza at al. \cite{KT86} gave a lemma about the $F'$-saturated graph with $F'$ has a center $v^*$.

%\begin{lemma}\cite{KT86}
%	Let $F=F'\setminus \left\{v^*\right\}$ and suppose that some vertex $v^*\in V(G)$ as degree $d(v^*)=n-1$. Then $G$ is $F'$-saturated if and only if $G\setminus{v^*}$ is $F$-saturated.
%\end{lemma} 

%For a fixed graph $F$, let $v^*$ be the vertex of $K_1$ and $F'=k_1\lor F$, then we call $F'$ is obtained from $F$ by adding a center $v^*$. Alex Cameron and Puleo \cite{CP22} believe that the upper bound about $sat(n,F')$ is the same spirit as the above lemma.

%\begin{lemma}\cite{CP22}
%	If $F'$ is obtained from $F$ by adding a  center $v^*$, then for all $n\ge|V(F')|$, we have $sat(n,F')\le(n-1)+sat(n-1,F)$. 
%\end{lemma}

 %Bollob\'{a}s \cite{B78} proved a lower bound of $2$-connected graph $G$ of order $n$.

\begin{lemma}\label{le4.1}\emph{(\cite{B78})}
	Let $G$ be  a $2$-connected graph of order $n$ with $diam(G)=2$. Then $e(G)\ge 2n-5$.
\end{lemma}

 We now restate two properties of saturation number with two graphs $F'$ and $G$ for which $G$ and $F'$ have  a center vertex $v^{*}$ and $v^*_0$, respectively. K\'{a}szonyi and Tuza \cite{KT86} observed the following  property.

\begin{lemma}\label{le2.4}\emph{(\cite{KT86}) }
	Let $F=F'\setminus \{v^*\}$.  Then $G$ is $F'$-saturated if and only if $G\setminus\{v^*\}$ is $F$-saturated.
\end{lemma} 

\noindent Conversely, for a fixed graph $F$, let $F'=K_1\lor F$, set $v^*$ as the specified vertex $K_1$. Clearly,  $F'$ has a center vertex $v^*$. Cameron and Puleo \cite{CP22} believe that the upper bound about $\sat(n,F')$ is the same spirit as the above lemma.

\begin{lemma}\label{le2.5}\emph{(\cite{CP22})}
	If $F'$ is obtained from $F$ by adding a  center $v^*$, then for all $n\ge|V(F')|$, we have $\sat(n,F')\le(n-1)+\sat(n-1,F)$. 
\end{lemma}
 
 \noindent Chen et al. \cite{CFFGJM17} demonstrated that $\sat(n,F)$ for a  linear forest $F=P_{k_1}\cup P_{k_2}\cup \cdots\cup P_{k_t}$ is determined by the smallest path in the forest.
 
 \begin{lemma}\label{le2.6}\emph{(\cite{CFFGJM17})}
 	If $F=P_{k_1}\cup P_{k_2}\cup \cdots \cup P_{k_t}$ where $k_1\ge k_2\ge \cdots \ge k_t$, $q=( \sum_{i=1}^t{k_t} )-1$, then 
 	$$
 	\sat( n,F ) =\left\{ \begin{array}{l}
 		n-\lfloor \frac{n}{a_{k_t}} \rfloor +c( n ) \ if\ k\ne 4\\
 		n-\lfloor \frac{n}{2} \rfloor +c( n ) \ if\ k=4\\
 	\end{array} \right.,
 	$$
 	for some constant $c(n)$ such that $0\le c( n ) \le \tbinom{q}{2} -q+\lceil \frac{q}{a_{k_t}} \rceil$.
 	
 \end{lemma} 
{\bf Proof of Theorem \ref{saturationjoin}:} Let $G$ be a $(K_1\lor F)$-saturated graph of order $n$ with the minimum number of edges. We first show the following property of $G$.

\vspace{2mm}
{\bf Claim 1}
	$diam(G)=2$.

\begin{proof}
	We observe that $G$ is not a complete graph because $G$ does not contain a copy of $K_1\lor F$. Hence, $diam(G)\ge 2$. It suffices to show that for any pair of nonadjacent vertices $u,v\in V(G)$, $d_G(u,v)\le 2$.  Note that $G+uv$ creates a copy of $K_1\lor F$ containing $uv$ according to our assumption. If $u$  and $v$ both belong to the copy of $F$ in $G+uv$, then they have a common neighbor in $G$. If one of $u$ and $v$ belongs to the copy of $F$, say $v$, then $u$ is regarded as  $K_1$ in the  copy of $K_1\lor F$. It follows that $u$ and $v$ have a common neighbor.   So we get $diam(G)\le 2$.
\end{proof}

We now consider the case $\varDelta ( G) =n-1$. Evidently, $G$ contains a center vertex, say $v^*$. Let $G'=G-v^*$. Observe that $G'$ is an $F$-saturated by Lemma \ref{le2.4}. It follows from the minimality of saturated graphs that
\begin{equation}\label{GG}
	\sat(n,K_1\lor F)=e(G)=n-1+e(G')\ge (n-1)+\sat(n-1,F).
\end{equation}
On the other hand,  Lemma \ref{le2.5} infers that the opposite of (3) also holds. Together with (\ref{GG}),  the proof is done.

  We next consider the remaining case $\varDelta (G) \le n-2$. With  the condition of maximum degree we first show that $G$ has the following property.
  
\vspace{2mm}
{\bf Claim 2}
	$G$ is $2$-connected.

\begin{proof}
	Assume to the contrary that $G$ is a $1$-connected graph. So there is some vertex of $G$, say $v$,
 such that  $G-v$ is disconnected.  Label the components of $G-v$ as $C_1,C_2,\ldots,C_t$, where $t\ge 2$. By Claim 2, $d_G(u_i,u_j)=2$ with $u_i \in C_i$ and $u_j \in C_j$ for $1\le i<j\le t$. Then $v\in N(u_i)\cap N(u_j)$. It can be inferred that  $d(v)=n-1$, a contradiction. Therefore, the claim is true.
\end{proof}

By Claims 1 and 2 and    Lemmas \ref{le4.1}, \ref{le2.4}, \ref{le2.5} and \ref{le2.6},  we have 
\begin{equation*}
(n-1)+\sat(n-1,F)\ge e(G)\ge 2n-5>(n-1)+\sat(n-1,F),
\end{equation*} a contradiction.

We now show the family of all extremal graphs. Suppose that $G$ is a minimum $K_1\lor F$-saturated graph. We now conclude that $G$ has a center vertex, say $v_0$.  Thus, we deduce that, from the argument above, $G$ is a minimum $K_1\lor F$-saturated graph if and only if $G\setminus\{v_0\}$ is $F$-saturated. Hence, all extremal graphs are characterized. In other words,  $H$ is an extremal graph of $F$ if and only if $K_1\lor H$ is an extremal graph of $K_1\lor F$.

Therefore, we complete the proof.\qed

\section{Concluding remarks}
In the paper, motivated by the fact that $T_k$ is $\{K_3,P_k\}$-saturated, we first study the minimum $\{K_3,P_k\}$-saturated tree. Based on it, we  determine $\sat(n,\{K_3,P_k\})$ with $k\ge 10$. Furthermore, $\sat(n,\{K_3,P_k\})$ can be used to bound the saturation number of $K_3\cup P_k$ as proposed in Relation (\ref{twobounds}).
Although we do not obtain the exact value of $\sat(n, K_3\cup P_k)$, we firmly believe that the upper bound in Relation (\ref{twobounds}) is indeed its saturation number. Hence, we pose the following problem.
\begin{problem}
For $k\ge 10$ is $\sat(n,K_3\cup P_k)=6+\sat(n,\{K_3,P_k\})$ true?
\end{problem}
\noindent In addition, we show the saturation number of the join of an isolated vertex and a linear forest without isolated vertices, which confirms Problem \ref{p1} for the specified graph. 

Note that  Lemma \ref{saturatedtree} obtains the minimum $\{K_3,P_k\}$-saturated tree for $k\ge 10$. 
We conclude this section by discussing the property for  small $k\le 9$. 

\vspace{2mm}
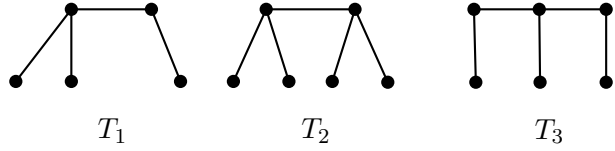
\begin{figure}[h]\label{5}
\begin{center}
\begin{picture}(223.3,49)\linethickness{0.8pt}
\Line(21,49)(0,21.7)
\Line(21,49)(21,21.7)
\Line(21,49)(51.3,49)
\put(31,.1){$T_1$}
\Line(51.3,49)(62.3,21.7)
\put(108,.1){$T_2$}
\Line(94.7,49)(103.3,21.7)
\Line(94.7,49)(82.3,21.7)
\Line(94.7,49)(128.3,49)
\Line(128.3,49)(140.7,21.4)
\Line(128.3,49)(119.7,21.7)
\Line(173.3,49)(174,21.4)
\Line(173.3,49)(198,49)
\put(196.3,0){$T_3$}
\Line(198,49)(198.3,21.4)
\Line(198,49)(223.3,49)
\Line(223.3,49)(223.3,21.4)
\put(21,49){\circle*{5}}
\put(21,21.7){\circle*{5}}
\put(0,21.7){\circle*{5}}
\put(62.3,21.7){\circle*{5}}
\put(51.3,49){\circle*{5}}
\put(103.3,21.7){\circle*{5}}
\put(82.3,21.7){\circle*{5}}
\put(94.7,49){\circle*{5}}
\put(140.7,21.4){\circle*{5}}
\put(119.7,21.7){\circle*{5}}
\put(128.3,49){\circle*{5}}
\put(174,21.4){\circle*{5}}
\put(173.3,49){\circle*{5}}
\put(198.3,21.4){\circle*{5}}
\put(198,49){\circle*{5}}
\put(223.3,21.4){\circle*{5}}
\put(223.3,49){\circle*{5}}
\end{picture}
\end{center}
\vspace{-2mm}
\caption{Graphs $T_1$, $T_2$, and $T_3$.}%\label{5}
		\end{figure}
\begin{proposition}\label{saturationnumbertree2}
Let $T_i$ be the graph shown in Figure 5 for $i=1,2,3$. Suppose that $T$ is a $\{K_3,P_k\}$-saturated tree that is not a star for $5\le k\le 9$. Then $T_0\subseteq T$, where 
 \begin{equation*}
T_0\cong \left\{ \begin{array}{ll}
	T_1& \emph{if}\ k=5,\\
	T_2\ \emph{or}\ T_3& \emph{if}\ k=6,\\
	T_{k}^{0}& \emph{if}\ 7\le k\le 8,\\
	T_{9}^{1}\ \emph{or}\ T_{9}^{0}& \emph{if}\ k=9.
\end{array} \right.      
 \end{equation*}
\end{proposition}
 \proof Note that  $T_0$ is $\{K_3,P_k\}$-saturated for $k\ge 7$ by Lemma \ref{pksaturated}. It is also true   for $k\in \{5,6\}$ by directly checking. Suppose that $T$ is a $\{K_3,P_k\}$-saturated tree.  We can assume $k\le 9$ by Lemma \ref{saturatedtree}.  Observe that $diam(T)\ge 3$ by our assumption that $T$ is not a star. Hence, $k\ge 5$. For the case $7\le k\le 9$, we directly obtain $T_0$ as claimed by the fact that $T_k^1\subseteq T$ and $T_k^0\subseteq T$ from the proof of  Lemma \ref{saturatedtree}. 
 
 So we now consider the remaining case $k\in \{5,6\}$. Recall that  $k-3\le diam(T) \le k-2$ for $k\ge 5$ by Lemma \ref{saturatedtree}. If $k=5$, then $diam(T)=3$. Let $P_4=v_1v_2v_3v_4$ be a longest path of $T$. Observe that $T+v_1v_4$ contains a copy of $P_5$, which infers that $v_2$ or $v_3$ has at least one neighbor. It follows that $T_0\subseteq T$. We now assume that $k=6$, then $diam(T)\in\{3,4\}$. If $diam(T)=3$, then $T$ has a longest path, say $P'_4=w_1w_2w_3w_4$. Observe that $T+w_1w_4$ contains a copy of $P_6$. It implies that $w_2$ and  $w_3$ have at least one neighbor, respectively. It follows that $T_0\subseteq T$. We now assume $diam(T)=4$. Hence, let $P_5=u_1u_2u_3u_4u_5$ be a longest path of $T$. Note that $T+u_1u_4$ creates a copy of $P_6$. We thus deduce that $u_3$ has a neighbor, which infers that $T_0\subseteq T$. By direct checking, $T_0$ is the unique tree with order $6$ in the case. \qed
 
In fact, we can also determine $\sat(n,\{K_3,P_k\})$ for $5\le k\le9$ by using the same way in the proof of Theorem \ref{saturationnumber2} together with Preposition \ref{saturationnumbertree2}. Due to the lack of a unified form, these cases are omitted here.

\bigskip% \vspace{5pt}

 %\noindent {\bf Declaration of competing interest}
 
%The authors declare that they have no known competing financial interests or personal relationships that could have appeared to influence the work reported in this paper.
%\vspace{5pt}

%\bigskip

\noindent {\bf Acknowledgments}
			
			%The authors would like to thank the editor and referees for their valuable suggestions which helps us to modify the presentation of the paper.
 Qing Cui was partially supported by the National Natural Science Foundation of China (Nos.
12171239 and 12271251). Erfei Yue is partially supported by ERC Advanced grant GeoScape (No. 882971). Shengjin Ji was partially supported by  Natural Science Foundation of Shandong Province, China 
(No. ZR2022MA077) and by Postgraduate Education Reform Project of Shandong Province, China (No. S-DYKC2023107).
\bigskip			
			%\vspace{2mm}

% \noindent{\bf Data availability}   
 
%No data was used for the research described in the article.        
%\noindent{\bf Conflict of interest}
			
			%The authors declare that they have no conflict interest regarding the publication of this paper.

\end{document}